\documentclass[11pt]{amsart}
\usepackage{a4wide}

\usepackage{latexsym}
\usepackage{amssymb}
\usepackage{amsmath}
\usepackage{latexsym}
\usepackage[all]{xy}
\usepackage{amscd}
\usepackage{dsfont}
\usepackage{graphics}
\usepackage{graphicx}
\usepackage{bbold}

\newcommand{\Z}{\mathbb{Z}}

\newtheorem{theorem}{Theorem}[section]

\newtheorem{lemma}[theorem]{Lemma}

\newtheorem{proposition}[theorem]{Proposition}
\newtheorem{corollary}[theorem]{Corollary}

\newtheorem{remark}[theorem]{Remark}

\DeclareMathOperator{\Image}{Im}
\DeclareMathOperator{\Ker}{Ker}
\DeclareMathOperator{\sdc}{SDC}
\DeclareMathOperator{\res}{res}
\DeclareMathOperator{\infl}{inf}
\DeclareMathOperator{\tr}{tr}
\DeclareMathOperator{\Aut}{Aut}
\DeclareMathOperator{\Out}{Out}

\begin{document}
\title[Seven-term exact sequence]{A seven-term exact sequence for the cohomology of a group extension}

\author[K. Dekimpe]{Karel Dekimpe}
\address{Katholieke Universiteit Leuven\\
Campus Kortrijk\\
8500 Kortrijk\\
Belgium}
\email{karel.dekimpe@kuleuven-kortrijk.be}
\email{sarah.wauters@kuleuven-kortrijk.be}

\author[M. Hartl]{Manfred Hartl}
\author[S. Wauters]{Sarah Wauters}
\address{Universit\'{e} de Valenciennes\\
  Valenciennes \\
  France}
\email{manfred.hartl@univ-valenciennes.fr}
\date{\today}
\thanks{S.\ Wauters is supported by a Ph.~D.~fellowship of the Research Foundation - Flanders (FWO)}
\thanks{Research supported by the Research Fund K.U.Leuven}

\begin{abstract}
In this paper we construct a seven-term exact sequence involving the cohomology groups of a group extension. Although the existence 
of such a sequence can be derived using spectral sequence arguments, there is little knowledge about some of the maps occuring in the 
sequence, limiting its usefulness. Here we present a construction using only very elementary tools, always related to 
the notion of conjugation in a group. This results in a complete and usable description of all the maps, which we describe
both on cocycle level as on the level of the interpretations of low dimensional cohomology groups (e.g.~group extensions).
\end{abstract}

\maketitle

\section{Introduction}
A classical tool to study the cohomology of groups fitting in a group extension 
\[\xymatrix{1 \ar[r] & N \ar[r] & G \ar[r] & Q \ar[r] & 1}\] 
is the Lyndon-Hochschild-Serre spectral sequence, which relates the cohomology of $G$ to the cohomology of the kernel $N$ and the quotient $Q$. 
For example, as observed in \cite{sah74-1}, for any $G$--module $M$, the Lyndon-Hochschild-Serre spectral sequence gives rise to an exact sequence 
\[\xymatrix{0 \ar[r] & H^1(Q,M^N) \ar[r]^-{\infl} & H^1(G,M) \ar[r]^-{\res} & H^1(N,M)^Q \ar[r]^-{\tr} & H^2(Q,M^N) }\]
\[\xymatrix{\mbox{} \ar[r]^-{\infl} & H^2(G,M)_1 \ar[r]^-{\rho} & H^1(Q,H^1(N,M)) \ar[r]^-{\lambda} & H^3(Q,M^N)},\] 
where $ H^2(G,M)_1 $ is the kernel of the restriction map $\res : H^2(G,M) \rightarrow H^2(N,M)$. The inflation and the restriction maps in the sequence are well understood, but the others are induced by differentials 
in the spectral sequence and there is no explicit description available, except from Huebschmann's description of $\lambda$, see below. At least, using cocycle manipulations, two different ad-hoc constructions of a map $\xymatrix{\tr:H^1(N,M)^Q \ar[r] & H^2(Q,M^N})$  were given in \cite{guic80-1} and  \cite{rous06-1} which render the left-hand 5-term sequence exact.

In this paper, we give an alternative, purely group theoretic construction of connecting maps $\tr$ and $\rho$ as above fitting into this sequence, taking the last map $\lambda$ to be the map constructed by  Huebschmann in \cite{hueb81-2}. Just as for the latter, our constructions are
based on mostly well-known  interpretations of the low dimensional cohomology groups, such as derivations, semi-direct complements, extensions and crossed modules. As a striking fact, our maps $\tr$ and $\rho$ are extracted from nothing but the conjugation action of group extensions of $G$ by $M$ which are split at least  over $N$, on the set of semi-direct complements of $N$: the map $\tr$ is basically given by the isotropy groups of this action, while the map $\rho$ more precisely encodes this action itself, modulo conjugation by elements of $M$. It is interesting to observe that Huebschmann's map $\lambda$ can be viewed as a broadening of the context from conjugation to  ``conjugation-like'' automorphisms.

We also give an explicit description of the three maps $\tr,\, \rho$, and $\lambda$ on the cocycle level. In particular, it turns out that the two constructions of a suitable map $\tr$ in \cite{guic80-1} and  \cite{rous06-1} both are just different descriptions of  our map $\tr$.

In general, we do not know whether or not the maps in our sequence coincides with the ones obtained from the spectral sequence. However, when $M$ is $N$-invariant, we know that $\tr$ equals the boundary map $d^{0,1}_2:H^1(N,M)^Q \rightarrow H^2(Q,M^N)$ (see lemma \ref{lem_evens_ding}). Furthermore, 
Huebschmann showed in \cite{hueb81-2} that his map $\lambda$ is the same as the one induced by the spectral sequence. 

We start by recalling two important techniques in our construction, namely the pull-back and a kind of push-out construction. In what follows, we will find a more appropriate way to regard $H^1(G,M)$, namely using semi-direct complements and/or splittings, instead of derivations (section \ref{der_sec_sdc}). To avoid complicating things, and to gain optimal insight in the matter, we use some notions of category theory, which are briefly overviewed in section \ref{hom_point_view}. We are then ready to construct the different maps in the sequence (section \ref{first_map}, \ref{second_map} and \ref{third_map}). It will be clear that the seven-term exact sequence is natural with respect to the modules, but we will still have to show that it is natural with respect to the group extension. We do this in section \ref{naturality_sequence}. In section \ref{Evens_dinges}, we show that, at least under some conditions, $\tr$ equals $d^{0,1}_2$. An explicit cocycle description of the maps $\tr, \, \rho$ and $\lambda$ is given in section \ref{cocycle_description}. Finally, we demonstrate our results for the Heisenberg groups in section \ref{Heisenberg}. We finish with some observations concerning split group extensions (section \ref{split_extension}). 

\section{Pull-back and push-out constructions}\label{pb_po}
The pull-back and the push-out are two key notions from category theory. The reason that they are also important to us, is that we can use them to describe induced maps on the extension level. 

Given two group morphisms $p_1 : E \rightarrow H$ and $p_2 : G \rightarrow H$, we can take the group $P$ consisting of all couples $(e,g) \in E \times G$ such that $p_1(e)=p_2(g)$, with the group law inherited from the direct product. There are natural maps $q_1 : P \rightarrow E$ and $q_2 : P \rightarrow G$, which are restrictions of the projection maps. We say that $(P, q_1, q_2)$ is the \emph{pull-back} of $p_1$ and $p_2$. When the morphisms are clear from the context, we sometimes just say that $P$ is the pull-back.
The universal property of the pull-back says that if there are two group morphisms $\alpha_1 : E' \rightarrow E$ and $\alpha_2 : E' \rightarrow G$ such that $p_1 \circ \alpha_1 = p_2 \circ \alpha_2$, there exists a unique group morphism $\alpha : E' \rightarrow P$ such that $q_1 \circ \alpha = \alpha_1$ and $q_2 \circ \alpha= \alpha_2$.
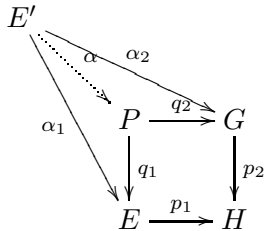
\[\xymatrix{E' \ar[rrd]^{\alpha_2} \ar@{.>}[rd]^{\alpha} \ar[rdd]_{\alpha_1} & & \\& P \ar[r]^{q_2} \ar[d]^{q_1} & G \ar[d]^{p_2}\\
& E \ar[r]^{p_1} & H}\]
If there is another group $P'$ and morphisms $q'_1 : P' \rightarrow E$ and $q'_2 : P' \rightarrow G$ such that $p_1 \circ q'_1=p_2 \circ q'_2$, having the universal property, it is not difficult to check that $P \cong P'$. In fact, the universal property can be used to define the pull-back, as is done in category theory. 

Now suppose we have a short exact sequence of groups 
\[\xymatrix{0 \ar[r] & M \ar[r]^{i_1} & E \ar[r]^{p_1} & H \ar[r]& 1}\] with abelian kernel,  and a group morphism $p_2 : G \rightarrow H$. Taking the pull-back $(P,q_1,q_2)$ of $p_1$ and $p_2$, we can find a map $i_2 : M \rightarrow P$ such that we obtain a commutative diagram with exact rows
\[\xymatrix{\underline{e'}: 0 \ar[r] & M \ar@{=}[d] \ar[r]^{i_2} & P \ar[d]^{q_1} \ar[r]^{q_2} & G \ar[d]^{p_2}\ar[r] & 1 \\
\underline{e}: 0 \ar[r] & M \ar[r]^{i_1} & E \ar[r]^{p_1} & H \ar[r] & 1.}\]

One can check that $[\underline{e'}]$ in the diagram is the image of $[\underline{e}]$ under the induced map $(p_2)^* : H^2(H,M) \rightarrow H^2(G,M)$, where $M$ is a $G$-module via $p_2$.

The push-out construction we want to use in this paper is the same as the one given by C.C.~Cheng and Y.C.~Wu (\cite{cw81-1}). 
Take a short exact sequence of groups with abelian kernel \[\xymatrix{0 \ar[r] & M_1 \ar[r]^{i_1} & E_1 \ar[r]^{p_1} & G \ar[r]& 1},\] a $G$-module $M_2$ and a $G$-module morphism $i_2:M_1 \rightarrow M_2$. We want to describe the map $(i_2)_* : H^2(G,M_1) \rightarrow H^2(G,M_2)$ on the level of extensions. 

Throughout this paper we will denote the action of an element $g\in G$ on an element $m$ of a $G$-module $M$ by $g\cdot m$.
There is an $E_1$-module structure on $M_2$ induced by $p_1$, i.e. $^{e_1}m_2=p_1(e_1) \cdot m_2$, so we can consider the semi-direct product $M_2\rtimes E_1$.
Set $E = (M_2 \rtimes E_1) / S$, where $S$ is the normal subgroup of $M_2 \rtimes E_1$ consisting of the elements of the form $(-i_2(m_1), i_1(m_1))$ for $m_1 \in M_1$. There are maps $j_1: E_1 \rightarrow E$ and $j_2 : M_2 \rightarrow E$, defined by taking the composition of the respective embeddings in $M_2 \rtimes E_1$ and the quotient map $M_2 \rtimes E_1 \rightarrow E$. We call $(E,j_1,j_2)$ the \emph{push-out construction} of the maps $i_1$ and $i_2$. Sometimes we omit the maps and say that $E$ is the push-out construction of $i_1$ and $i_2$, if it is clear what $j_1$ and $j_2$ are.
 Observe that $j_1 \circ i_1 = j_2 \circ i_2$.

There is a ``universal property'' of the push-out construction. Let $E'$ be a group and let $\rho_1 : E_1 \rightarrow E'$ and $\rho_2 : M_2 \rightarrow E'$ be group homomorphisms such that $\rho_1 \circ i_1=\rho_2 \circ i_2$. There exists a homomorphism $h: E \rightarrow E'$ with $h \circ j_1=\rho_1$ and $h \circ j_2=\rho_2$, if and only if 
\begin{equation}\label{voorwaarde_uitbreiding}\rho_2(e_1 \cdot m_2)=\rho_1(e_1) \rho_2(m_2) \rho_1(e_1)^{-1}.\end{equation} 
In this case, $h$ is unique. 
\[\xymatrix{M_1 \ar[d]_{i_2} \ar[r]^{i_1} & E_1 \ar[ddr]^{\rho_1} \ar[d]_{j_1}&\\ 
						M_2 \ar[drr]_{\rho_2} \ar[r]^{j_2} & E \ar@{.>}[dr]^{h \ \ \ \ }\\
						    &   & E'}\]
Observe that we can give an easier description of condition (\ref{voorwaarde_uitbreiding}), defining an action of $E_1$ on $E'$ by conjugation, i.e. ${}^{e_1} x=\rho_1(e_1) x \rho_1(e_1)^{-1}$ for $x \in E'$ and $e_1 \in E_1$. Now we can replace condition (\ref{voorwaarde_uitbreiding}) by demanding that $\rho_2$ is compatible with the action of $E_1$. 

If we take $\rho_1=j_1$ and $\rho_2=j_2$, equation (\ref{voorwaarde_uitbreiding}) holds. Moreover, the universal property determines $E$ up to isomorphism.

Take maps $p_1: E_1 \rightarrow G$ and $1: M_2 \rightarrow G$, the trivial map, and observe that we can find can a map $p : E \rightarrow G$ such that $p \circ j_1=p_1$ and $p \circ j_2=1$. 
Thus, there is a commutative diagram
\[\xymatrix{0 \ar[r] & M_1 \ar[d]^{i_2}\ar[r]^{i_1} & E_1 \ar[d]^{j_1} \ar[r]^{p_1} & G \ar[r]& 1\\0 \ar[r] & M_2 \ar[r]^{j_2} & E \ar[r]^p & \ar@{=}[u] G \ar[r]& 1.}\] The lower sequence is automatically exact. One can check that the image of the class of the upper sequence under the natural map $H^2(G,M_1) \rightarrow H^2(G,M_2)$, induced by $i_2$, can be represented by the lower sequence. 

\section{Derivations, splittings and semi-direct complements}\label{der_sec_sdc}

Let $G$ be a group and $M$ a $G$-module. It will be useful for the construction of the exact sequence to have different descriptions of the first cohomology group $H^1(G,M)$. Consider the standard split extension of $G$ by $M$
\[\xymatrix{0 \ar[r] & M \ar[r]^{i_0 \ \ \ } & M \rtimes G \ar[r]^{\ \ \ p_0} & G \ar[r] & 1}.\] It is well-known that $H^1(G,M)$ is isomorphic to the group $\mbox{Der}(G,M)/\mbox{Inn} (G,M)$, where $\mbox{Der}(G,M)$ is the group of derivations $d : G \rightarrow M$, and $\mbox{Inn} (G,M)$ are the inner derivations. We can associate to each derivation $d$ a splitting $s : G \rightarrow M \rtimes G$, $s(g)=(d(g),g)$. In this way, we get a bijection between $\mbox{Der}(G,M)/\mbox{Inn} (G,M)$ and $\mbox{Sec}(G,M) / \sim_M$, where $\mbox{Sec}(G,M)$ is the set of splittings of the split exact sequence. The equivalence relation is the following: $s_1 \sim_M s_2$ if there exists $m \in M$ such that $s_1=i_0(m)s_2 i_0(m)^{-1}$. These interpretations of $H^1(G,M)$ are well-known, and appear in several textbooks on cohomology of groups, for example in \cite{brow82-1}. 

For our purposes, there is a more convenient way to look at $H^1(G,M)$. Define the set of semi-direct complements $\mbox{SDC}(G,M)$ of $G$ in 
$M\rtimes G$ as the set containing all subgroups $H \leq M \rtimes G$, such that the restriction of $p_0$ to $H$ is an isomorphism onto $G$. The map $s \mapsto s(G)$ entails a bijection between $\mbox{Sec}(G,M)$ and $\mbox{SDC}(G,M)$. Now we can transfer the equivalence relation $\sim_M$ to a relation on $\mbox{SDC}(G,M)$, namely $H \sim H'$ if there exists an $m \in M$ such that $H=i_0(m) H' i_0(m)^{-1}$. There is also an induced group structure 
\[H_1 + H_2 = \{(h_1+h_2, g) \in M \rtimes G \ | \ (h_1,g) \in H_1, (h_2,g) \in H_2\},\] that turns $\sdc (G,M)$ and its quotient under the equivalence relation into abelian groups.

Consider the case where there is an exact sequence of groups 
\[\xymatrix{1 \ar[r] & N \ar[r] & G \ar[r]^{\pi} & Q \ar[r] & 1}.\] We know that conjugation in the normal subgroup $N$ induces a $G$-module structure on $H^1(N,M)$. Moreover, the action of $G$ factors through $Q$, so that $H^1(N,M)$ becomes a $Q$-module. We would like to know what the $G$-action looks like on derivations, splittings or semi-direct complements, since this action directly gives us the action of $Q$. 

We denote the action of an element $g \in G$ on a derivation $d : N \rightarrow M$ by $^g d$. Using the standard $N$-resolution, one can see that 
\[(^gd)(n)=g \cdot d(g^{-1} n g).\] It is now straightforward to check that  on the level of splittings $s: N \rightarrow M \rtimes N$ and
semi-direct complements $H$ of $N$ in $M \rtimes N$ the above action translates to 
\[(^gs)(n)=(0,g) s(g^{-1}ng)(0,g)^{-1}\mbox{ and } ^g H = (0,g) H (0,g)^{-1},\]  where we view $M \rtimes N$ as a subgroup of $M \rtimes G$. If $d$ is an inner derivation, $^g d$ will also be an inner derivation. Therefore, we have a well-defined action of $G$ on $\mbox{Der}(N,M)/\mbox{Inn} (N,M)$,  $\mbox{Sec}(N,M) / \sim_M$ and $\sdc (N,M)/ \sim$. These actions factor through $Q$ and correspond to the usual $Q$-module structure on $H^1(N,M)$.

\section{A categorical point of view}\label{hom_point_view}

Let $Pair$ be the category of all pairs $(G,M)$ where $G$ is a group and $M$ is a $G$-module. A morphism $(\alpha, f)$ from $(G,M)$ to $(G',M')$ consists of a group homomorphism $\alpha : G' \rightarrow G$ and a $G'$-module morphism $f : M \rightarrow M'$, where $M$ is a $G'$-module via $\alpha$. Now $\sdc (-,-)$ is a functor from $Pair$ to the category of sets, where the induced maps are defined by 
\[\sdc (\alpha, f)(H)=\{(f(m),g)\ | \ (m,\alpha(g)) \in H \}, \]
for $H \in \sdc (G,M)$.
 Note that $\sdc (1,f)$ and $\sdc (\alpha, 1)$ are given by the direct and inverse image, respectively. 

As we will see, the functor $\sdc (G,-)$ preserves products and final objects. 
This is an important property, since it implies that the functor will preserve group objects. Let $\mathcal{C}$ be a category that has products and a final object $1_{\mathcal{C}}$. We call $\mathcal{C}$ a \emph{category that admits group objects}. A group object $X$ is an object such that there exist morphisms $\mu : X \times X \rightarrow X$ (``{\em multiplication}''), $\eta : 1_{\mathcal{C}} \rightarrow X$ (``{\em unit}'') and $i : X \rightarrow X$ (``{\em inverse}'') satisfying the commutative diagrams 
that correspond to the usual group axioms (e.g.\ see \cite[page 75]{macl71-1}). 

If $F : \mathcal{C} \rightarrow \mathcal{D}$ is a functor between two categories that admit group objects, there is a well known criterion for $F$ to preserve group objects.
\begin{lemma}\label{preserving_group_objects}
If $F : \mathcal{C} \rightarrow \mathcal{D}$ preserves products and final objects, then $F$ will preserve group objects.
\end{lemma}
 Explicitly, if $(X, \mu, \eta, i)$ is a group object, $(FX, F \mu \circ h^{-1}, F \eta, F i)$ will be a group object, where $h$ is the canonical 
map $h=(F(pr_1),F(pr_2)): F(X\times X)\rightarrow FX \times FX $, which is an isomorphism since $F$ preserves products. Often, we will omit $h^{-1}$
and simply write $(F X, F \mu, F \eta, F i)$.

\begin{lemma}\label{homomorphism_of_group_objects}
Let $\mathcal{C}$ and $\mathcal{D}$ be categories that admit group objects, and let $F : \mathcal{C} \rightarrow \mathcal{D}$ and $G : \mathcal{C} \rightarrow \mathcal{D}$ be functors that preserve group objects. 
If $a$ is a natural transformation between $F$ and $G$, then for every group object $X$ of $\mathcal{C}$, $a_X : FX \rightarrow GX$ is a homomorphism of group objects, i.e.\ the diagram 
\[\xymatrix{FX \times FX \ar[d]^{F \mu} \ar[r]^{a_X \times a_X} & GX \times GX \ar[d]^{G \mu}\\
FX \ar[r]^{a_X} & GX}\]
commutes.
\end{lemma}

This is also well known. We will use this lemma to show that certain maps in our seven-term sequence are homomorphisms. 

Now let's turn our attention to the functor $\sdc(G,-)$ from the category $G$--modules to the category of sets. 
\begin{lemma}\label{sdc_preserves_products}
The functor $\sdc(G,-)$ preserves products and the final object.
\end{lemma} The proof is left to the reader. 

An object $M$ in the category of $G$-modules is always a group object, with the obvious commutative group law, denoted as $+$. It follows that $\sdc (G,M)$ is a group object in the category of sets, with multiplication $\sdc (1, +)$. The reader can check that this is the same group law as the one described in section \ref{der_sec_sdc}.

\section{Construction of $tr$}\label{first_map}

Given a short exact sequence of groups
\[\xymatrix{1 \ar[r] & N \ar[r]^j & G \ar[r]^{\pi} & Q \ar[r] & 1}\] and a $G$-module $M$, we want to construct an exact sequence \[\xymatrix{0 \ar[r] & H^1(Q,M^N) \ar[r]^{\infl} & H^1(G,M) \ar[r]^{\res} & H^1(N,M)^Q \ar[r]^{\tr} & H^2(Q,M^N)\ar[r]^{\infl} & H^2(G,M)}.\] Note that $M$ is also an $N$-module, and $M^N$ is a $Q$-module, so the cohomology groups are well-defined. The existence of the exact sequence follows from the Hochschild-Serre spectral sequence. In this section, we want to give an explicit description of a map that can be chosen to be the third map.

It is known that the image of the second inflation map is contained in the kernel of the restriction map $\res : H^2(G,M) \rightarrow H^2(N,M)$ (or see corollary \ref{cor:im_in_1}). So it is not surprising that we will first turn our attention to a general construction involving extensions in $\Ker(\res)$. 
Take an extension
\begin{equation}\label{extension}\underline{e}: \xymatrix{0 \ar[r] & M \ar[r]^i & E \ar[r]^p & G \ar[r] & 1}\end{equation} which is partially split. This means that the sequence
\[\xymatrix{0 \ar[r] & M \ar[r]^{i \ \ \ } & p^{-1}(N) \ar[r]^{\ \ p} & N \ar[r] & 1}\] is split and therefore equivalent to the standard split extension through an isomorphism $\gamma : M \rtimes N \rightarrow p^{-1}(N)$. (We will sometimes identify the two extensions.) In other words, the class $[\underline{e}]$ belongs to the kernel $H^2(G,M)_1$ of the restriction map $\res: H^2(G,M) \rightarrow H^2(N,M)$.

A \emph{partial semi-direct complement} $H$ of $N$ in $\underline{e}$ is a subgroup $H \leq E$ such that $p(H) = N$ and $p|_H : H \rightarrow N$ is an isomorphism. Two partial semi-direct complements $H_1$ and $H_2$ of $N$ in $\underline{e}$ are \emph{equivalent} (denoted by $H_1 \sim H_2$) if there exists an element $m \in M$ such that $H_2=i(m)H_1i(m)^{-1}$. 
Every partially split extension $\underline{e}$ determines an action of $E$ on the partial semi-direct complements $H \leq E$ of $N$ in $\underline{e}$, induced by conjugation. %, so 
It is immediate that this action will factor through $p : E \rightarrow G$, after passing to equivalence classes. It factors further through $\pi : G \rightarrow Q$. 
The obtained action of $Q$ is given by ${}^q [H]=[{}^e H]$, where $\pi(p(e))=q$.

As a special case, we can consider the standard split extension of $G$ by $M$
\[\underline{e}_0:\xymatrix{0 \ar[r] & M \ar[r] & M \rtimes G \ar[r] & G \ar[r] & 1.}\] In this situation, the set of partial semi-direct complements corresponds exactly to $\sdc(N,M)$ and the action we obtain then is the same action as the one we discussed at the end of section~\ref{der_sec_sdc}. 

Every partial semi-direct complement $H$ of $N$ in $\underline{e}$ determines a homomorphism $s : N \rightarrow E$, mapping $n$ to its unique pre-image under $p$ in $H$. Observe that $p \circ s = \mbox{id}$. Such a homomorphism $s : N \rightarrow E$ is called a \emph{partial splitting} of $\underline{e}$ over $N$, and for every $H$, there is a unique partial splitting $s$ with $s(N)=H$.
Observe that ${}^e s$ defined by ${}^e s(n) = e s(p(e)^{-1}n p(e))e^{-1}$ is a partial splitting with  $({}^e s) (N)={}^e (s(N))$, so this is an action of $E$ on the partial splittings of $\underline{e}$ that corresponds to the above action of $E$ on the partial semi-direct complements. Two partial splittings $s_1$ and $s_2$ of $\underline{e}$ are \emph{equivalent} (we write $s_1 \sim s_2$) if there exists an element $m \in M$ such that ${}^{i(m)}s_1=s_2$ or equivalently, $i(m) s_1(n) i(m)^{-1}=s_2(n)$ for all $n \in N$. 

Now we can start with the construction of a map $\omega$ that will give rise to a map $\tr : H^1(N,M)^Q \rightarrow H^2(Q,M^N)$.  
Take an extension (\ref{extension}) and a partial semi-direct complement $H$ of $N$ in $\underline{e}$. Since $H$ is isomorphic to $N$, one can expect that ``taking the quotient of $E$ and $H$'' will correspond to taking the quotient of $G$ and $N$. Of course, $H$ doesn't need to be a normal group of $E$, so we have to pass to the normalizer $N_E(H)$ in $E$. 
The following two lemmas are easily checked. 
\begin{lemma}\label{lemma_1}
The intersection $i(M) \cap N_E(H)$ equals $i(M^N)$.
\end{lemma}

\begin{lemma}\label{lem:surjective_invariant}
The restriction $p|_{N_E(H)}: N_E(H) \rightarrow G$ is surjective iff $^{e}H \sim H$ for all $e \in E$ (or equivalently, $^{e}s \sim s$ for all $e \in E$, where $s$ is the unique splitting with $s(N)=H$).
\end{lemma}

Let $\Omega$ be the set of all pairs $(\underline{e},H)$, where $\underline{e}$ is an extension
\[\underline{e}:\xymatrix{0 \ar[r] & M \ar[r]^i & E \ar[r]^p & G \ar[r] & 1,}\] 
$[\underline{e}] \in H^2(G,M)_1$ and $H \leq E$ is a partial semi-direct complement of $N$ in $\underline{e}$ with ${}^e H \sim H$ for all $e \in E$.
By the lemmas above, the sequence
\[\xymatrix{0 \ar[r] & M^N \ar[r]^-{i} & N_E(H) \ar[r]^-{p} & G \ar[r] & 1}\] is exact. 
It is not difficult to check that this sequence induces an exact sequence
\[\underline{e}':\xymatrix{0 \ar[r] & M^N \ar[r]^-{\overline{i}} & N_E(H)/H \ar[r]^-{\overline{p}} & Q \ar[r] & 1,}\] where $\overline{i}$ and $\overline{p}$ are the induced maps. 
This gives a map 
\[\omega: \Omega \rightarrow H^2(Q,M^N),\]
mapping $(\underline{e},H)$ to the class of the extension $\underline{e}'$.

First we make some remarks about the construction of $\omega$. In the commutative diagram
\[\xymatrix{0 \ar[r] & M \ar[r]^i  & E \ar[r]^p & G \ar[r] & 1\\
0 \ar[r] & M^N \ar@{^{(}->}[u]^{j} \ar@{=}[d] \ar[r]^{i \ \ } & N_E(H) \ar[u]^{\iota} \ar[d] \ar[r]^{\ \ p} & G \ar@{=}[u] \ar[d]^{\pi}\ar[r] & 1\\
0 \ar[r] & M^N \ar[r]^{\overline{i} \ \ \ } & N_E(H)/H \ar[r]^{\ \ \ \overline{p}}& Q \ar[r] & 1,}\]
$E$ is isomorphic to the push-out construction of the inclusion map $j:M^N \hookrightarrow M$ and $i:M^N \rightarrow N_E(H)$, and $N_E(H)$ is isomorphic to the pull-back of $\overline{p}:N_E(H)/H \rightarrow Q$ and $\pi:G \rightarrow Q$. 
To show the first statement, let $E'$ be the push-out construction of $j:M^N \rightarrow M$ and the map $i:M^N \rightarrow N_E(H)$. 
We know from section \ref{pb_po} that we can find a map $\widetilde{p}$ such that
\[\xymatrix{0 \ar[r] & M \ar[r]^{\widetilde{i}} & E' \ar[r]^{\widetilde{p}} & G \ar[r]& 1 \\ 0 \ar[r] & M^N \ar[u]^{j} \ar[r] & N_E(H) \ar[u]^q \ar[r]^{p} & G \ar@{=}[u]\ar[r]& 1}\]
is a commutative diagram with exact rows. It is not difficult to see that we can use the universal property of the push-out construction to get a map $h : E' \rightarrow E$ such that $h \circ \widetilde{i} = i$ and $h \circ q = \iota$. Thus, we get the following commutative diagram with exact rows. 
\[\xymatrix{0 \ar[r] & M \ar[r]^{\widetilde{i}} & E' \ar[d]^h \ar[r]^{\widetilde{p}} & G \ar[r]& 1 \\ 0 \ar[r] & M \ar@{=}[u] \ar[r]^i & E  \ar[r]^p & G \ar@{=}[u]\ar[r]& 1}\] 
The right-hand side commutes, since $\widetilde{p}\circ \widetilde{i}=1=(p \circ h)\circ\widetilde{i}$, by definition of $\widetilde{p}$. Also, $(p \circ h) \circ q = p \circ \iota = \widetilde{p} \circ q$. Therefore uniqueness in the universal property gives us the equality $\widetilde{p}=p \circ h$. Now by the five-lemma, $h$ is an isomorphism and $E \cong E'$. The pull-back statement is proven in the same way.

Our goal is to find a map $H^1(N,M)^Q \rightarrow H^2(Q,M^N)$ completing the five-term exact sequence. 
Define $\sdc(N,M)^Q$ as the pre-image of $H^1(N,M)^Q$ under the projection $\sdc(N,M)$ $  \rightarrow H^1(N,M)$. As we have seen, $\sdc(N,M)^Q$ consists of all partial semi-direct complements $H$ of $N$ in $ M \rtimes G$ with ${}^{(m,g)}H \sim H$ for all $(m,g) \in M \rtimes G$. Therefore, we can define $\hat{\tr} : \sdc(N,M)^Q \rightarrow H^2(Q,M^N)$ as \[\hat{\tr}(H)=\omega(\underline{e}_0,H),\] where $\underline{e}_0$ is the standard split extension of $G$ by $M$. 
We will show that this indeed induces a map $\tr:H^1(N,M)^Q \rightarrow H^2(Q,M^N)$ that can be chosen as the third map in the five-term sequence. For this, $\tr$ has to be a homomorphism, and the corresponding five-term sequence has to be exact. We also want the map to be natural. 

\begin{lemma}\label{nat_trans}
$\hat{\tr}: \sdc(N,-)^Q \rightarrow H^2(Q,-^N)$ is a natural transformation of functors.
\end{lemma}
\begin{proof}
We have to show that for every $G$-module map $f : M_1 \rightarrow M_2$, the following diagram is commutative.
\[\xymatrix{\sdc(N,M_1)^Q \ar[r]^{\hat{\tr}} \ar[d]^{\sdc(1,f)} & H^2(Q,M_1^N) \ar[d]^{H^2(1,f)}\\
\sdc(N,M_2)^Q \ar[r]^{\hat{\tr}}  & H^2(Q,M_2^N) }\]
For $H \in \sdc(N,M_1)^Q$, use the push-out construction to find a representant of $f_*(\hat{\tr}(H))$. The universal property and the five-lemma give an equivalence of $f_*(\hat{\tr}(H))$ with the extension $\hat{\tr}(f_*(H))$, in the same way as in the remarks about the construction of $\omega$.
\end{proof}

It is immediate that the cohomology functors $H^n(G,-)$ preserve products
and the final object $0$. Furthermore, if $M_1$ and $M_2$ are $G$-modules, the action of $G$ on $\sdc(G,M_1 \times M_2) \cong \sdc(G,M_1) \times \sdc(G,M_2)$ is the diagonal action. 
Using this information and lemma \ref{sdc_preserves_products}, we see that also the functors $\sdc(N,-)^Q$ and $H^2(Q,-^N)$ preserve products and final objects, so they will preserve group objects. 
Now, by lemma \ref{homomorphism_of_group_objects} and lemma \ref{nat_trans}, we obtain the following result.
\begin{proposition}
The map $\hat{\tr} : \sdc(N,M)^Q \rightarrow H^2(Q,M^N)$ is a homomorphism.
\end{proposition}

To see that $\hat{\tr}$ induces a well-defined map $\tr:H^1(N,M)^Q \rightarrow H^2(Q,M^N)$, we need the following lemma. 
\begin{lemma} \label{lemma2}
Given $H \in \sdc(N,M)^Q$. There exists an $\widetilde{H} \in \sdc (G,M)$ such that $\widetilde{H} \cap (M \rtimes N) = H$ iff $\hat{\tr}(H)=0$.
\end{lemma}
\begin{proof}
Suppose that $ H = \widetilde{H} \cap (M \rtimes N)$. Since $M \rtimes N$ is normal in $M \rtimes G$, $H$ is normal in $\widetilde{H}$, so $N_{M \rtimes G}(H)$ contains the semi-direct complement $\widetilde{H}$ of $G$. This means that the short exact sequence
\[\xymatrix{0 \ar[r] & M^N \ar[r] & N_{M \rtimes G}(H) \ar[r] & G \ar[r] & 1}\]
splits, and consequently, $\hat{\tr}(H)=0$. Conversely, choose a semi-direct complement $\widetilde{Q} \subseteq N_{M \rtimes G}(H)/H$ of $Q$ and take the inverse image under $q : N_{M \rtimes G}(H) \rightarrow N_{M \rtimes G}(H)/H$. It is not difficult to prove that this is the required semi-direct complement $\widetilde{H}$ in $\sdc(G,M)$.
\end{proof}
There are some immediate consequences of this lemma.
\begin{corollary}
The map $\tr: [H] \mapsto \hat{\tr}(H)$ is well-defined.
\end{corollary}

\begin{proof}
We can extend every semi-direct complement of the form ${}^{i(m)}H_0$ to the semi-direct complement ${}^{i(m)}\widetilde{H}_0$ of $M \rtimes G$, where $H_0 = \{(0,n) \in M\rtimes G \ | \ n \in N\}$ and $\widetilde{H}_0 = \{(0,g) \in M \rtimes G \ | \ g \in G\}$.
\end{proof}

\begin{corollary}
The sequence 
\[\xymatrix{0 \ar[r] & H^1(Q,M^N) \ar[r]^-{\infl} & H^1(G,M) \ar[r]^-{\res} & H^1(N,M)^Q \ar[r]^-{\tr} & H^2(Q,M^N)}\]
is exact. 
\end{corollary}
We will show that $\tr:H^1(N,M)^Q \rightarrow H^2(Q,M^N)$ is a map that fits in the five-term exact sequence. To complete the proof of exactness, we need the following lemma.
\begin{lemma} \label{lemma3}
Let $\underline{e}: \xymatrix{0 \ar[r] & M \ar[r]^i & E \ar[r]^p & G \ar[r]& 1}$ be an extension of $G$ by $M$. 
\begin{itemize}
\item If $\underline{e}$ is partially split and there exists a partial semi-direct complement $H$ of $N$ in $\underline{e}$ with $^{e}H \sim H$ for all $e \in E$, then $\infl(\omega(\underline{e},H))=[\underline{e}]$. 
\item Conversely, if there exists an $[\underline{e'}] \in H^2(Q,M^N)$ such that $[\underline{e}]=\infl[\underline{e'}]$, then there exists a partial semi-direct complement $H$ of $N$ in $\underline{e}$ with $^{e}H \sim H$ for all $e \in E$, such that  $[\underline{e'}]=\omega(\underline{e},H)$.
\end{itemize}
\end{lemma}
\begin{proof}
The inflation map $\infl:H^2(Q,M^N) \rightarrow H^2(G,M)$ is the composition of the maps $H^2(Q,M^N) \rightarrow H^2(G,M^N)$ and $H^2(G,M^N) \rightarrow H^2(G,M)$, induced by respectively the projection map $\pi: G \rightarrow Q$ and the inclusion map $M^N \hookrightarrow M$. Using section \ref{pb_po}, we see that for any extension $\underline{e}'$ of $M^N$ and $Q$, we can represent $\infl(\underline{e}')$  by the lower row in the diagram
\begin{equation}\label{double_diagram}\xymatrix{\ \ \ \ \ \underline{e}' : 0 \ar[r] & M^N \ar[r]^{i_0} & E_0 \ar[r]^{p_0} & Q \ar[r] & 1\\
\ \ \ \ \ \ \ \ \ \ 0 \ar[r] & M^N \ar[r]^{i'} \ar@{=}[u] \ar[d]^{j} & P \ar[u] \ar[d] \ar[r]^{p'} & G \ar[r] \ar@{=}[d] \ar[u]^{\pi}& 1\\
\infl(\underline{e}') : 0 \ar[r] & M \ar[r]^{i}  & E \ar[r]^{p} & G \ar[r] & 1.}\end{equation}
Here $P$ is the pull-back of $p_0$ and $\pi$, and $E$ is the push-out construction of $i'$ and $j$.

Now it follows from the remarks after the definition of $\omega$ that $\infl(\omega(\underline{e},H))=[\underline{e}]$ when $\underline{e}$ is partially split and $H$ is a partial semi-direct complement of $N$ in $\underline{e}$ with ${}^e H \sim H$. This proves the first part of the lemma. 

For the second part, suppose that $[\underline{e}]=\infl[\underline{e}']$, where 
\[\underline{e}' : \xymatrix{0 \ar[r] & M^N \ar[r]^{i_0} & E_0 \ar[r]^{p_0} & Q \ar[r] & 1.}\] We show that we can find a partial semi-direct complement $H$ of $N$ in $\underline{e}$ with ${}^e H \sim H$ for all $e \in E$, such that the extensions $\underline{e'}$ and the canonical representative of $\omega(\underline{e}, H)$ are equivalent. 
Since $[\underline{e}]=\infl[\underline{e}']$, we have again a diagram (\ref{double_diagram}) 
where the upper right square is a pull-back diagram and the lower left square is a push-out construction. 
Using properties of the pull-back (see section \ref{pb_po}), we see that the short exact sequence $\xymatrix{1 \ar[r] & N \ar[r]^j & G \ar[r]^{\pi} & Q \ar[r] & 1}$ 
induces a short exact sequence 
$\xymatrix{1 \ar[r] & N \ar[r]^{\alpha} & P \ar[r]^{h} & E_0 \ar[r] & 1}.$ 
Define $\widetilde{H}=\alpha(N) \lhd P$ and observe that $h$ induces an isomorphism $\overline{h}$ of extensions as follows.
\[\xymatrix{0 \ar[r] & M^N \ar[r]^{i_0} \ar@{=}[d]  & E_0  \ar[r]^{p_0} & Q \ar@{=}[d] \ar[r] & 1\\
0 \ar[r] & M^N \ar[r]^{\overline{i'}}  & P / \widetilde{H} \ar[u]^{\overline{h}}_{\cong} \ar[r]^{\overline{p'}} & Q \ar[r] & 1}\] 
If we set $H = \gamma(\widetilde{H}) \subset E$, we see that $p(H)=
j(N)$, and a similar argument shows that $p|H : H \rightarrow N$ is an isomorphism. 
Therefore $H$ is a partial semi-direct complement of $N$ in $\underline{e}$ and $\underline{e}$ is partially split. Observe also that $\gamma(P) \subset N_E(H)$, so $p(N_E(H))=G$. From lemma \ref{lem:surjective_invariant} it follows that ${}^e H \sim H$ for all $e \in E$. 
 
We can use lemma \ref{lemma_1} and the five-lemma to see that $\gamma$ is an isomorphism of extensions 
\[\xymatrix{0 \ar[r] & M^N \ar[r]^{i'} \ar@{=}[d]  & P \ar[d]^{\gamma}_{\cong} \ar[r]^{p'} & G \ar@{=}[d] \ar[r] & 1\\
0 \ar[r] & M^N \ar[r]^{i}  & N_E(H)  \ar[r]^{p} & G \ar[r] & 1.}\] 
Since $\gamma(\widetilde{H})=H$, the map induces an isomorphism $\overline{\gamma}: P / \widetilde{H} \rightarrow N_E(H)/H$, 
such that there is an equivalence of extensions 
\[\xymatrix{0 \ar[r] & M^N  \ar@{=}[d] \ar[r]^{i_0} & E_0 \ar[d]^{\overline{\gamma} \circ \overline{h}^{-1}}_{\cong} \ar[r]^{p_0} & Q \ar@{=}[d]\ar[r] & 1\\
0 \ar[r] & M^N \ar[r] & N_E(H) / H \ar[r] & Q \ar[r] & 1.}\] The lower extension represents $\omega(\underline{e},H)$, so this proves that $[\underline{e}']=\omega(\underline{e},H)$.
\end{proof}

This proves the exactness of the sequence 
\[\xymatrix{0 \ar[r] & H^1(Q,M^N) \ar[r]^-{\infl} & H^1(G,M) \ar[r]^-{\res} & H^1(N,M)^Q \ar[r]^-{\tr} & H^2(Q,M^N)\ar[r]^-{\infl} & H^2(G,M)},\] since $\tr[H]=\omega(\underline{e}_0,H)$ and $\infl(\underline{e}_0)=0$. 
Furthermore, we have the following corollaries, which are important for the next section. 
\begin{corollary}\label{cor:im_in_1}
The image of $\infl$ is contained in $H^2(G,M)_1$.
\end{corollary}

\begin{corollary}\label{cor:image_inf}
Let $[\underline{e}] \in H^2(G,M)_1$. Now $[\underline{e}] \in \Image \infl$ iff there exists a partial semi-direct complement $H$ of $N$ in $\underline{e}$ with ${}^e H \sim H$ for all $e \in E$ (or equivalently, a partial splitting $s : N \rightarrow E$ of $\underline{e}$ over $N$ with ${}^e s \sim s$ for all $e \in E$).
\end{corollary}

\section{Construction of $\rho$}\label{second_map}
We want to construct a map $\rho : H^2(G,M)_1 \rightarrow H^1(Q,H^1(N,M))$, extending the five-term sequence to a six-term exact sequence. Take an extension $[\underline{e}] \in H^2(G,M)_1$. This is an extension of the form
\[\underline{e} : \xymatrix{0 \ar[r] & M \ar[r]^i & E \ar[r]^p & G \ar[r] & 1,}\] 
for which the induced extension 
$ \xymatrix{0 \ar[r] & M \ar[r]^-i & p^{-1}(N) \ar[r]^-p & N \ar[r] & 1} $
 is split. 
Choose a partial splitting $s_0 : N \rightarrow E$ of $\underline{e}$ over $N$. We define a map
\[\widetilde{\rho}_0(\underline{e}) : E \rightarrow \mbox{Der}(N,M)\] associated to $s_0$, sending $x \in E$ to the derivation $d_x : N \rightarrow M$, defined by $i(d_x(n))=({}^x s_0)(n)s_0(n)^{-1}$. 
Another way to describe this derivation is by taking the derivation associated to $^{x}H_0=\gamma^{-1}(x\gamma(H_0)x^{-1})$, where $H_0=\{(0,n) \ | \ n \in N\}$ and $\gamma : M \rtimes N \rightarrow p^{-1}(N)$ is the isomorphism given by $\gamma(m,n)=i(m)s_0(n)$. 

We know from the previous section that the action of $E$ on the equivalence classes of partial splittings factors through $G$ and $Q$, so this will also be the case for $\widetilde{\rho}_0(\underline{e})$. Hence, we obtain a map \[\overline{\rho}_0(\underline{e}): Q \rightarrow \mbox{Der}(N,M)/\sim \ \cong H^1(N,M),\] sending $q$ to the class of the derivation $\widetilde{\rho}_0(\underline{e})(x)$, with $\pi(p(x))=q$. Observe that $\overline{\rho}_0(\underline{e})$ is trivial iff ${}^x s_0 \sim s_0$ for all $x \in E$.  

We now show that the map $\overline{\rho}_0(\underline{e})$ is a derivation, where $H^1(N,M)$ has the natural $Q$-module structure as discussed in section \ref{der_sec_sdc}. 
\begin{lemma}
The map $\overline{\rho}_0(\underline{e}): Q \rightarrow H^1(N,M)$ is a derivation.
\end{lemma}
\begin{proof}
It suffices to show that $\widetilde{\rho}_0(\underline{e})$ is a derivation, where the action of $E$ on $\mbox{Der}(N,M)$ is given via $p : E \rightarrow G$. Take $x$, $y \in E$ and define $d_x=\widetilde{\rho}_0(\underline{e})(x)$, $d_y=\widetilde{\rho}_0(\underline{e})(y)$ and $d_{xy}=\widetilde{\rho}_0(\underline{e})(xy)$. Observe that $i(({}^g d)(n))=\tilde{g} i(d(g^{-1} n g)) \tilde{g}^{-1}$, where $p(\tilde{g})=g$. As a result, 
\begin{eqnarray*} i(({}^{p(x)}d_y)(n)) &=& i(d_{xy}(n))i(d_x(n))^{-1}.\end{eqnarray*}
This means that $\widetilde{\rho}_0(\underline{e})$ and $\overline{\rho}_0(\underline{e})$ are derivations.
\end{proof}

For two different partial splittings $s_0$ and $s_1$ of $\underline{e}$ over $N$, there always exists a derivation $d\in \mbox{Der}(N,M)$ such that $s_1(n)=i(d(n)) s_0(n)$ for all $n\in N$. In this case, we write $s_1=ds_0$.
\begin{lemma}
Let $s_1=d  s_0$. Then ${}^x s_1 = {}^{p(x)}d \ {}^x s_0 $ for all $x \in E$.
\end{lemma}
The proof is left to the reader. 

\begin{lemma} \label{welldefinednextmap}
Take $s_0$ and $s_1$, partial splittings of $\underline{e}$ over $N$ with $s_1=ds_0$ and take the associated maps  $\widetilde{\rho}_0(\underline{e})$, $\widetilde{\rho}_1(\underline{e}): E \rightarrow \mbox{Der}(N,M)$. 
Then for all $x \in E$, $\widetilde{\rho}_1(\underline{e})(x)-\widetilde{\rho}_0(\underline{e})(x)=
{}^{p(x)}d - d$.
\end{lemma}
\begin{proof}
Set $d_1=\widetilde{\rho}_1(\underline{e})(x)$ and $d_0=\widetilde{\rho}_0(\underline{e})(x)$. Using the previous lemma and the fact that $i(d_1(n))=({}^x s_1)(n) s_1(n)^{-1}$, we see that $i(d_1(n))=i(({}^{p(x)}d)(n)) ({}^x s_0)(n)s_1(n)^{-1}$. Now it is clear that 
$i(d_1(n))=i(({}^{p(x)}d)(n))i(d_0(n))i(-d(n))$ and the result follows.
\end{proof}
As a consequence, the maps $\underline{e} \mapsto [\widetilde{\rho}_0(\underline{e})] \in H^1(E,\mbox{Der}(N,M))$ and $\underline{e} \mapsto [\overline{\rho}_0(\underline{e})] \in H^1(Q,H^1(N,M))$ are independent of the choice of the partial splitting. 
Thus we can define a map $\hat{\rho}$, mapping a partially split extension $\underline{e}$ of $G$
by $M$ to $\hat{\rho}(\underline{e})=[\overline{\rho}_0(\underline{e})]$.
We show that this induces a well-defined homomorphism of groups
\[\begin{array}{rccl} 
\rho :  & H^2(G,M)_1 & \rightarrow & H^1(Q,H^1(N,M));\\
 & [\underline{e}] & \mapsto & [\overline{\rho}_0(\underline{e})]. \end{array}\]

First, we prove a more general result. 
\begin{lemma}
If $M_1$ and $M_2$ are two $G$-modules and $\alpha : M_1 \rightarrow M_2$ is a $G$-module homomorphism, then $\alpha_* (\hat{\rho}(\underline{e}))=\hat{\rho}(\underline{e}')$, where $\alpha_*$ is the induced map $H^1(Q,H^1(N,M_1)) \rightarrow H^1(Q,H^1(N,M_2))$ and $\underline{e}'$ is an extension of $G$ by $M_2$, that fits in a commutative diagram 
\[\xymatrix{\underline{e}: 0 \ar[r] & M_1 \ar[d]^{\alpha}\ar[r]^{i} & E \ar[d]^{\beta} \ar[r]^{p} & G \ar[r]& 1\\ \underline{e}': 0 \ar[r] & M_2 \ar[r]^{i'} & E' \ar[r] & \ar@{=}[u] G \ar[r]& 1.}\]
\end{lemma}
Observe that the last statement is equivalent (modulo canonical isomorphisms) with the fact that $(E',\beta,i')$ is the push-out construction of the inclusion map $i$ of $\underline{e}$ and the map $\alpha$, or with $[\underline{e}']=H^2(\mathbb{1}, \alpha)[\underline{e}]$.
One can prove this equivalence using the methods in the remarks about the construction of $\omega$, following lemma \ref{lem:surjective_invariant}.
\begin{proof}
Take a partial splitting $s_1 : N \rightarrow E$ of $\underline{e}$ over $N$ and the
 partial splitting $s_2 = \beta \circ s_1$ of $\underline{e}'$ over $N$. Observe that $\beta({}^x s_1 (n))={}^{\beta(x)} s_2 (n) $ for all $x \in E$ and $n \in N$.

Since $\overline{\rho}_1(\underline{e})$ maps $q$ to the class of $\widetilde{\rho}_1(\underline{e})(x)$ in $H^1(N,M_1)$, with $\pi(p(x))=q$, the image of $\rho(\underline{e})$ under the map $\alpha_*$ will be the class of the derivation in $\mbox{Der}(Q,H^1(N,M_2))$ that maps $q$ to $[\alpha \circ \widetilde{\rho}_1(\underline{e})(x)]$. Also, $i'\Big((\alpha \circ \widetilde{\rho}_1(\underline{e})(x))(n)\Big)=i'\Big(\widetilde{\rho}_2(\underline{e}')(\beta(x))(n)\Big)$. It follows that $\alpha \circ \widetilde{\rho}_1(\underline{e})=\widetilde{\rho}_2(\underline{e}') \circ \beta$, and one easily sees that this implies that $\hat{\rho}(\underline{e}')=\alpha_* (\hat{\rho}(\underline{e}))$.
\end{proof}

We state some immediate consequences.
\begin{corollary}
The map $\rho : H^2(G,M)_1  \rightarrow  H^1(Q,H^1(N,M))$ is well-defined.
\end{corollary}
\begin{corollary}
The map $\rho$ is natural with respect to the modules.
\end{corollary}

The following theorem is now easily proven using lemma \ref{homomorphism_of_group_objects} and the fact that the cohomology functors $H^n(B,-)$ preserve products for every group $B$. 
\begin{theorem}
The map $\rho : H^2(G,M)_1 \rightarrow H^1(Q,H^1(N,M))$ is a homomorphism.
\end{theorem}

We want to prove that the definition of $\rho$ yields an exact sequence. 
\begin{lemma}
The sequence 
\[\xymatrix{\ldots \ar[r] & H^1(N,M)^Q \ar[r]^{\tr} & H^2(Q,M^N) \ar[r]^{\infl} & H^2(G,M)_1 \ar[r]^{\rho \ \ \ \ \ \ } & H^1(Q,H^1(N,M))}\]
is exact.
\end{lemma}
This follows from corollary \ref{cor:image_inf} and lemma \ref{welldefinednextmap}. 

\section{Construction of $\lambda$}\label{third_map}
In this section, we give a map $\lambda : H^1(Q,H^1(N,M)) \rightarrow H^3(Q,M^N)$ that completes the  seven-term exact sequence 
\[\xymatrix{0 \ar[r] & H^1(Q,M^N) \ar[r]^-{\infl} & H^1(G,M) \ar[r]^-{\res} & H^1(N,M)^Q \ar[r]^-{\tr} & H^2(Q,M^N)& \\ \ \  & {} \ar[r]^-{\infl} & H^2(G,M) \ar[r]^-{\rho} & H^1(Q,H^1(N,M))\ar[r]^-{\lambda} & H^3(Q,M^N).}\]
We will describe the construction given by Huebschmann in \cite{hueb81-2}. Huebschmann proves that his construction yields the differential map induced by the spectral sequence, so the map will automatically have all the good properties. In particular it will be a homomorphism, that is natural in a stronge sense (see \cite{hueb81-2}).

We will make use of the following interpretation of the third cohomology group. It is known that $H^3(G,M)$ corresponds to equivalence classes of \emph{crossed extensions}, i.e. exact sequences of the form
\[\xymatrix{0 \ar[r] & M \ar[r] & C \ar[r]^{\delta} & \Gamma \ar[r] & G \ar[r] & 1},\]
where $\xymatrix{C \ar[r]^{\delta}& \Gamma}$ is a crossed module that induces the given action of $G$ on $M$ (see \cite[IV.5]{brow82-1}). The equivalence relation is generated by elements $(\mathbb{1}_{M}, \alpha : C \rightarrow C', \beta : \Gamma \rightarrow \Gamma', \mathbb{1}_G)$, such that the diagram 
\[\xymatrix{0 \ar[r] & M \ar[r] \ar@{=}[d] & C \ar[r]^{\delta} \ar[d]^{\alpha}& \Gamma \ar[r] \ar[d]^{\beta} & G \ar[r] \ar@{=}[d] & 1\\
0 \ar[r] & M \ar[r] & C' \ar[r]^{\delta'} & \Gamma' \ar[r] & G \ar[r] & 1
}\]
is commutative, and $(\alpha, \beta)$ is a homomorphism of crossed modules. This in particular means that $\alpha$ is compatible with the action of $\Gamma$, where $\Gamma$ acts on $C'$ via $\beta$.

Now let 
\[\underline{e}_0: \xymatrix{0 \ar[r] & M \ar[r] & M \rtimes N \ar[r] & N \ar[r] & 0}\]
be the standard split extension of $N$ by $M$, and define $\Aut(N,M)\subset \Aut(M) \times \Aut(N)$, the group of all couples of automorphisms $(\sigma,\phi) \in \Aut(M) \times \Aut(N)$ with $\sigma(n \cdot m)=\phi(n)\cdot \sigma(m)$. Observe that every element of $\Aut(N,M)$ induces an automorphism of $M \rtimes N$. Also, there is a map $\chi : G  \rightarrow \Aut(N,M)$ mapping $g$ to $(i_g,c_g)$, where $i_g(m)=g \cdot m$ and $c_g(n)=gng^{-1}$.

Take $\Aut^M(M \rtimes N) \subset \Aut(M \rtimes N)$, the group of all automorphism of $M \rtimes N$ that map $M$ to itself. Observe that these automorphism are of the form 
\[(m,n) \mapsto (\sigma(m)+d(n), \phi(n))\] where $(\sigma,\phi) \in \Aut(N,M)$ and $d \circ \phi^{-1} \in \mbox{Der}(N,M)$. There is an obvious homomorphism $\theta : \Aut^M(M \rtimes N) \rightarrow \Aut(N,M)$, which fits in a split exact sequence
\[\xymatrix{0 \ar[r] & \mbox{Der}(N,M) \ar[r]^-{\alpha} & \Aut^M(M \rtimes N) \ar[r]^-{\theta} & \Aut(N,M) \ar[r] & 1,}\] where $\alpha$ maps a derivation $d$ to the automorphism $(m,n) \mapsto (m + d(n),n)$.

Let $\Aut_G(\underline{e}_0)$ be the pull-back of $\chi$ and $\theta$.
Using the pull-back properties, we find a split exact sequence
\[\xymatrix{0 \ar[r] & \mbox{Der}(N,M) \ar[r]^-{i_1} & \Aut_G(\underline{e}_0) \ar[r]^-{p_1} & G \ar[r] & 1.}\]
Observe that the induced $G$-module structure on $\mbox{Der}(N,M)$ coincides with the one given in section \ref{der_sec_sdc}.
There is a splitting $s : G \rightarrow \Aut_G(\underline{e}_0)$, mapping $g$ to the couple $(c_{(0,g)},g)$, where $c_{(0,g)}$ is conjugation with $(0,g)$ in $M \rtimes G$. Since $N$ is a normal subgroup of $G$, $M \rtimes N$ will be a normal subgroup of $M \rtimes G$ such that $c_{(0,g)}$ can indeed be considered as an automorphism of $M \rtimes N$.

Take $\beta_1 : M \rtimes N \rightarrow \Aut^M(M \rtimes N)$ mapping $(m,n)$ to $c_{(m,n)}$, conjugation with $(m,n)$, and $\beta_2 :M \rtimes N \rightarrow G$, mapping $(m,n)$ to $n$. The pull-back property gives us a map $\beta : M \rtimes N \rightarrow \Aut_G(\underline{e}_0)$. Define an action of $\Aut_G(\underline{e}_0)$ on $M \rtimes N$ by setting $^{(h,g)} (m,n)=h(m,n)$. The reader can check that this turns $\beta : M \rtimes N \rightarrow \Aut_G(\underline{e}_0)$ into a crossed module.

It is easy to see that we obtain a commutative diagram
\[\xymatrix{0 \ar[r] & M \ar[r] \ar[d]^{-\delta^0} & M \rtimes N \ar[r] \ar[d]^{\beta} & N \ar[d] \ar[r] & 1\\
0 \ar[r]&\mbox{Der}(N,M) \ar[r]^{i_1} & \Aut_G(\underline{e}_0) \ar[r]^-{p_1} & G \ar[r] & 1,}\]
where $-\delta^0(m)(n)=m - n \cdot m$. Furthermore, the images of the vertical maps are normal subgroups of the groups in the bottom row, so we can take cokernels without loosing exactness of the rows, thanks to the snake-lemma and injectivity of the map $N\to G$. Define $\Out_G(\underline{e}_0) = \Aut_G(\underline{e}_0) / \Image \beta$. We get the following commutative diagram with exact rows. 
\begin{equation}\label{big_diagram}\xymatrix{0 \ar[r] & M \ar[r] \ar[d]^{-\delta^0} & M \rtimes N \ar[r] \ar[d]^{\beta} & N \ar[d] \ar[r] & 1\\
0 \ar[r]&\mbox{Der}(N,M) \ar[d] \ar[r]^{i_1} & \Aut_G(\underline{e}_0) \ar[d]^{\Pi}\ar[r]^-{p_1} & G \ar[d]^{\pi}\ar[r] & 1\\
0 \ar[r]&H^1(N,M) \ar[r]^{i_2} & \Out_G(\underline{e}_0) \ar[r]^-{p_2} & Q \ar[r] & 1}\end{equation}
The standard splitting of the first row is compatible with the given splitting $s$ of the second row, and the resulting quotient map $\overline{s} : Q \rightarrow \Out_G(\underline{e}_0)$ is a splitting of the third row. 

We restrict our attention to the second column of the diagram. It is not difficult to see that $M^N$ is the kernel of $\beta$, so we obtain a crossed extension
\[\underline{e}:\xymatrix{0 \ar[r] & M^N \ar[r] & M \rtimes N \ar[r]^-{\beta} & \Aut_G(\underline{e}_0) \ar[r]^-{\Pi}  & \Out_G(\underline{e}_0) \ar[r] & 1.}\] 

Take $D \in \mbox{Der}(Q,H^1(N,M))$ and define a new splitting $D\overline{s}: Q \rightarrow \Out_G(\underline{e}_0)$ as $D\overline{s}(q) = i_2(D(q))\overline{s}(q)$. 
Now \[\underline{e}_D : \xymatrix{0 \ar[r] & M^N \ar[r] & M \rtimes N \ar[r]^-{\beta} & \Pi^{-1}(D\overline{s}(Q)) \ar[r]^-{\Phi} & Q \ar[r] & 1}\] is a crossed extension, with $\Phi=p_2 \circ \Pi$. Observe that $\Pi^{-1}(D\overline{s}(Q))$ with the associated maps can be seen as the pull-back of $\Pi$ and $D\overline{s}$. One can check that this means that $[\underline{e}_D]=(D\overline{s})^* [\underline{e}]$.
It is straightforward to see that the induced action of $Q$ on $M^N$ coincides with the given one.
Define \[\hat{\lambda} : \mbox{Der}(Q,H^1(N,M)) \rightarrow H^3(Q,M^N):\; D \mapsto \hat{\lambda}(D)=[\underline{e}_D]. \]
Huebschmann proves in \cite{hueb81-2} that $\hat{\lambda}$ yields a well-defined homomorphism $\lambda:H^1(Q,H^1(N,M)) \rightarrow H^3(Q,M^N)$, 
coinciding with the corresponding differential of the spectral sequence.

Since we don't know whether or not our map $\rho$ is the same as the map obtained by Sah in \cite{sah74-1}, we still have to check exactness 
of the following part of the sequence
\[\xymatrix{\cdots \ar[r] & H^2(G,M)_1 \ar[r]^-{\rho} & H^1(Q,H^1(N,M)) \ar[r]^-{\lambda} & H^3(Q,M^N)}.\]
We will make use of the following result due to Huebschmann (\cite{hueb80-1}). 

\begin{lemma}
A crossed extension 
\[\xymatrix{0 \ar[r] & B \ar[r] & C \ar[r]^{\delta} & \Gamma \ar[r]^{\rho} & Q \ar[r] & 1}\]
 is equivalent to the zero extension if and only if there exists  a short exact sequence \break $\xymatrix{1 \ar[r]& C \ar[r]& E \ar[r]& Q \ar[r]& 1}$ and a homomorphism $h : E \rightarrow \Gamma$ such that the diagram
\[\xymatrix{ & 1 \ar[r]& C \ar@{=}[d]\ar[r]& E \ar[d]^{h}\ar[r]& Q \ar@{=}[d] \ar[r]& 1\\
0 \ar[r]& B \ar[r]& C \ar[r]^{\delta} & \Gamma \ar[r]^{\rho} & Q\ar[r] & 1}\]
is commutative and $(\mathbb{1}: C \rightarrow C,h : E \rightarrow \Gamma)$ is a homomorphism of crossed modules.
\end{lemma}
Now we can prove exactness.
\begin{lemma}
The sequence
\[\xymatrix{\cdots \ar[r] & H^2(G,M)_1 \ar[r]^-{\rho} & H^1(Q,H^1(N,M)) \ar[r]^-{\lambda} & H^3(Q,M^N)}\] is exact.
\end{lemma}
\begin{proof}
First suppose that $[D] \in H^1(Q,H^1(N,M))$ with $\lambda[D]=0$. This means that we can find a short exact sequence 
$\xymatrix{1 \ar[r]& M \rtimes N \ar[r]^-i & E \ar[r]^-p & Q \ar[r]& 1}$ and a homomorphism of crossed modules 
$(\mathbb{1}_{M\rtimes N},\,h : E \rightarrow \Pi^{-1}(\widetilde{Q}))$ such that the diagram
\begin{equation}\label{late_in_the_evening}
\xymatrix{ & 1 \ar[r]& M \rtimes N \ar@{=}[d]\ar[r]^{i}& E \ar[d]^{h}\ar[r]^{p}& Q \ar@{=}[d] \ar[r]& 1\\
0 \ar[r]& M^N \ar[r]& M \rtimes N \ar[r]^{\beta} & \Pi^{-1}(\widetilde{Q}) \ar[r]^-{\Phi} & Q\ar[r] & 1,}
\end{equation}
with $\widetilde{Q}=D\overline{s}(Q)$, is commutative. We want to find an extension of $G$ by $M$ that is partially split, such that $[D]$ is the image of the extension under $\rho$. 

Let $p': E \rightarrow G$ be the composition of $h$ with the map $p_1: \Aut_G(\underline{e}_0) \rightarrow G$ and observe that $\pi \circ p' = p$. The kernel of $p'$ is easily seen to be $M$, so we find an extension
\[\underline{e}:\xymatrix{0 \ar[r] & M \ar[r]^{i'} & E \ar[r]^{p'} & G \ar[r] & 1,}\] where $i'$ is the composition $M \hookrightarrow M \rtimes N \stackrel{i}{\rightarrow} E$. One can show that the induced action of $G$ on $M$ coincides with the given action. 
Clearly, $p'^{-1}(N)$ is isomorphic to $M \rtimes N$ through $i$. 
We show that this extension is the extension we need.

Write $[D_{\underline{e}}]=\rho[\underline{e}]$. We fix the partial splitting $s_0: N\hookrightarrow M\rtimes N \stackrel{i}{\rightarrow}E$ of $\underline{e}$ over $N$. We know that $D_{\underline{e}}(q)=[\widetilde{\rho}_0(\underline{e})(x)]$ for $\pi \circ p'(x)=q$ or equivalently, $p(x)=q$. Using the description of $\rho$ on semi-direct complements, we see that $\widetilde{\rho}_0(\underline{e})(x)$ is the derivation associated to the semi-direct complement ${}^{x}H_0=i^{-1}(xi(H_0)x^{-1})$. Here $H_0$ is the standard partial semi-direct complement of $N$ in the standard split extension. Since $h$ is a homomorphism of crossed modules, ${}^{x}H_0$ equals ${}^{h(x)}H_0$. 

From the commutative diagram (\ref{late_in_the_evening}), we deduce that $\Pi \circ h (x) = D \overline{s} \circ p (x)$, so \begin{equation} \label{eq:pi} \Pi \circ h (x) = \Pi (i_1(d) s(g))\end{equation} for $D(p(x))=[d]$ and $\pi(g)=p(x)$. As the reader can easily check, ${}^{\beta(m,n)}H_0 \sim H_0$, so it follows that the action of $\Aut_G(\underline{e}_0)$ on $[H_0]$ factors through $\Out_G(\underline{e}_0)$. 
Equation (\ref{eq:pi}) shows that \[{}^{x}H_0 \sim {}^{i_1(d) s(g)}H_0={}^{i_1(d) }H_0={}^{\alpha(d) }H_0,\] and the last one has associated derivation $d$. It follows directly that $[\widetilde{\rho}_0(\underline{e})(x)]=[d]=D(p(x))$, so $D_{\underline{e}}(p(x))=D(p(x))$. %This finishes the proof. 

Conversely, consider a partially split extension 
\[\underline{e}:\xymatrix{0 \ar[r] & M \ar[r]^{i'} & E \ar[r]^{p'} & G \ar[r] & 1}\]
and a given partial splitting $s_0 : N\rightarrow E$ of $\underline{e}$ over $N$. Let $D=\overline{\rho}_0(\underline{e})$ and let $\gamma : M \rtimes N \rightarrow p'^{-1}(N)$ 
be the isomorphism associated with $s_0$ as before. 
By defining $p=\pi \circ p'$ and observing that $\Ker p = \gamma(M \rtimes N)$, we obtain an exact sequence
\[\xymatrix{ & 1 \ar[r]& M \rtimes N \ar[r]^-{\gamma}& E \ar[r]^{p}& Q  \ar[r]& 1.}\] 

We construct a homomorphism $h : E \rightarrow \Aut_G(\underline{e}_0)$, from $h_1 = p': E \rightarrow G$ and $h_2 : E \rightarrow \Aut^M(M \rtimes N)$ sending $x$ to $h_2(x)$, with $h_2(x)(m,n)=\gamma^{-1}(x \gamma (m,n) x^{-1})$. 
The compositions of $h \circ \gamma$ with respectively  $\Aut_G(\underline{e}_0) \rightarrow G$ and $\Aut_G(\underline{e}_0) \rightarrow \Aut^M(M \rtimes N)$ equal the compositions of $\beta$ with these two maps. By the pull-back property, this means that $h \circ \gamma=\beta$. 

We know that $h(x)=i_1(d)s(p'(x))$ for some derivation $d$. Since ${}^{h(x)}H_0={}^{x}H_0$, it is clear that $d=\widetilde{\rho}_0(\underline{e}_0)(x)$. It follows that $\Pi \circ h(x)=D \overline{s}(p(x))$, so $h(E) \subset \Pi^{-1}(\widetilde{Q})$ with $\widetilde{Q}=D \overline{s}(Q)$ and we get a commutative diagram
\[\xymatrix{ & 1 \ar[r]& M \rtimes N \ar@{=}[d]\ar[r]^{\gamma}& E \ar[d]^{h}\ar[r]^{p}& Q \ar@{=}[d] \ar[r]& 1\\
0 \ar[r]& M^N \ar[r]& M \rtimes N \ar[r]^{\beta} & \Pi^{-1}(\widetilde{Q}) \ar[r] & Q\ar[r] & 1,}\] 
where the left square is a homomorphism of crossed modules. This shows that $\lambda[D]=0$ and the sequence is exact. 
\end{proof}

\begin{remark}
Using spectral sequence arguments, one can see that in case $H^2(N,M)^Q=0$, the  seven-term sequence can be extended to the following eight-term exact sequence: 
\[\xymatrix{0 \ar[r] & H^1(Q,M^N) \ar[r]^-{\infl} & H^1(G,M) \ar[r]^-{\res} & H^1(N,M)^Q \ar[r]^-{\tr} & H^2(Q,M^N)& \\  \ar[r]^-{\infl} & H^2(G,M) \ar[r]^-{\rho} & H^1(Q,H^1(N,M))\ar[r]^-{\lambda} & H^3(Q,M^N) \ar[r]^-{\infl} & H^3(G,M).}\]
\end{remark}

Since $\lambda$ coincides with the differential $d_2^{1,1}$ of the spectral sequence, this follows from spectral sequence arguments as introduced in \cite{hs53-1}. 

\section{Naturality of the sequence}\label{naturality_sequence}
We already know that the maps in the seven-term exact sequence are natural with respect to the modules. Here we show that the maps are also natural with respect to the short exact sequence of groups. Let 
\[\xymatrix{0 \ar[r] & N' \ar[d]^{\alpha_0} \ar[r] & G' \ar[d]^{\alpha}\ar[r]^{\pi'} & Q' \ar[d]^{\overline{\alpha}}\ar[r] & 1\\
0 \ar[r] &  N \ar[r] & G \ar[r]^{\pi} & Q \ar[r] & 1}\]
be a morphism of group extensions. Take a $G$-module $M$. Then $M$ is also a $G'$-module through $\alpha$. Observe that now automatically $M^N \subset M^{N'}$, and call the inclusion $j : M^N \rightarrow M^{N'}$. 

The naturality of $\lambda$ has already been shown in \cite{hueb81-2}. We show that $\tr$ and $\rho$ are also natural with respect to the short exact sequence of groups. 

To prove the naturality of $\tr$, one has to show the commutativity of 
\[\xymatrix{H^1(N,M)^Q \ar[r]^{\tr} \ar[d]^{\alpha_0^*} & H^2(Q,M^N) \ar[d]\\
H^1(N',M)^{Q'} \ar[r]^{\tr} & H^2(Q',M^{N'}).}\]
The right hand map is the composition of $\alpha^* : H^2(Q,M^N) \rightarrow H^2(Q',M^N)$ and $j_*: H^2(Q',M^N) \rightarrow H^2(Q',M^{N'})$.
Take $[d] \in H^1(N,M)^Q$. Then automatically $\alpha_0^*[d]=[d \circ \alpha_0] \in H^1(N',M)^{Q'}$, since ${}^{g'}(d \circ \alpha_0)={}^{\alpha(g')}d \circ \alpha_0$. 
To find the image of $\tr[d]$ under $j_* \circ \alpha^*$, we first take a pull-back and then a push-out of the sequence representing $[\underline{e}_d]=\tr[d]$.
This means there is a diagram
\[\xymatrix{\underline{e}_d: & 0 \ar[r] & M^N \ar[r] & N_{M\rtimes G}(H)/H \ar[r] & Q \ar[r] & 1 \\
\alpha^*(\underline{e}_d) : & 0 \ar[r] & M^{N} \ar[d]^{j}\ar@{=}[u] \ar[r] & P \ar[u] \ar[d] \ar[r] & Q' \ar[r]\ar[u]^{\alpha} & 1\\
\underline{e} : & 0 \ar[r] & M^{N'}  \ar[r] & E \ar[r] & Q' \ar@{=}[u] \ar[r] & 1 ,}\] where the right upper square is a pull-back diagram, while the left lower square is a push-out construction. As before, we let $s_d(n)=(d(n),n) \in M \rtimes G$ for all $n \in N$ and $s_d'(n')=(d(\alpha(n')),n') \in M \rtimes G'$ for all $n' \in N'$, and 
set $H=s_d(N)$ and $H'=s'_d(N')$.

Observe that there exists a commutative diagram with exact rows
\[\xymatrix{0 \ar[r] & M \ar[r]^-{i'} \ar@{=}[d]& M \rtimes G' \ar[d]^{\mathbb{1} \rtimes \alpha} \ar[r]^-{p'} & G' \ar[r] \ar[d]^{\alpha} & 1 \\
0 \ar[r] & M \ar[r]^-{i} & M \rtimes G \ar[r]^{p} & G \ar[r] & 1. }\] 
Furthermore, $(m,g) \in N_{M \rtimes G}(H)$ iff for all $n \in N$, ${}^g d (n) - d(n)=n \cdot m -m$, and $(m,g') \in N_{M \rtimes G'}(H')$ iff for all $n' \in N'$, ${}^{g'} (d \circ \alpha_0) (n') - (d \circ \alpha_0)(n')=\alpha(n') \cdot m -m$. Now it is easy to see that $(\mathbb{1} \rtimes \alpha)^{-1}(N_{M \rtimes G}(H)) \subseteq N_{M \rtimes G'}(H')$. Set $S = (\mathbb{1} \rtimes \alpha)^{-1}(N_{M \rtimes G}(H))$. Observe that  $i'(m) \in S$ iff $m \in M^N$, and $p'(S)=G'$, since $p : N_{M \rtimes G}(H) \rightarrow G$ is surjective.  This means that the lower sequence in the diagram
\[\xymatrix{0 \ar[r] & M^{N'} \ar[r] & N_{M\rtimes G'}(H') \ar[r] & G' \ar[r] & 1\\
0 \ar[r] & M^N \ar[r]^{i'} \ar[u]^{j} & S \ar[r]^{p'} \ar[u] & G' \ar@{=}[u]\ar[r] & 1}\] is exact.
It is also clear that $(\mathbb{1} \rtimes \alpha)(H')\subseteq H$, so $H' \subseteq S$ and even $H' \triangleleft S$, since $S \subset N_{M \rtimes G'}(H')$. It follows that we get an exact sequence 
\[\xymatrix{0 \ar[r] & M^N \ar[r]^-{\overline{i'}} & S / H' \ar[r]^-{\overline{p'}} & Q' \ar[r] & 1}\] and we claim that this sequence is equivalent to the sequence $\alpha^*[\underline{e}_d]$. 
Take $\phi_1: S / H' \rightarrow N_{M \rtimes G}(H)/H$, mapping $(m,g')H'$ to $(m,\alpha(g'))H$ and $\phi_2: S / H' \rightarrow Q'$, $\phi_2 = \overline{p'}$. It is easy to see that the maps are well-defined and that, by the universal property, we obtain a map $\phi : S /H' \rightarrow P$. 
Now one only has to check that the diagram 
\[\xymatrix{ & 0 \ar[r] & M^N \ar[r]^-{\overline{i'}} \ar@{=}[d] & S / H' \ar[d]^{\phi} \ar[r]^{\overline{p'}} & Q' \ar@{=}[d]\ar[r] & 1 \\
\alpha^*(\underline{e}_d) : & 0 \ar[r] & M^N \ar[r] & P \ar[r] & Q' \ar[r] & 1}\] is commutative, but that is not difficult.

Since the diagram 
\[\xymatrix{\underline{e}_d': & 0 \ar[r] & M^{N'} \ar[r] & N_{M \rtimes G'}(H')/ H' \ar[r] & Q' \ar[r] & 1 \\
& 0 \ar[r] & M^N \ar[r] \ar[u]^{j} & S / H' \ar[u]\ar[r] & Q' \ar@{=}[u] \ar[r] & 1}\]
is commutative, the upper row $\underline{e}_d'$, representing $\tr(\alpha^*[d])$, is the push-out construction of the lower row, equivalent to $\alpha^*(\underline{e}_d)$. 
It follows that $[\underline{e}]=\tr(\alpha^*[d])$, so the map $\tr$ is natural. 

To prove that $\rho$ is natural, we have to show that 
\begin{equation}\label{diagr}\xymatrix{H^2(G,M)_1 \ar[r]^-{\rho} \ar[d]^{\alpha^*} & H^1(Q,H^1(N,M)) \ar[d]^{\alpha^*}\\
H^2(G',M)_1 \ar[r]^-{\rho} & H^1(Q',H^1(N',M))}\end{equation} commutes. 
Observe that the right hand map is the composition 
\[\xymatrix{H^1(Q,H^1(N,M)) \ar[r]^{\overline{\alpha}^*} & H^1(Q',H^1(N,M))\ar[r]^{H^1(\mathbb{1},\alpha_0^*)}& H^1(Q',H^1(N',M)).}\]Take $[\underline{e}] \in H^2(G,M)_1$, with 
\[\xymatrix{\underline{e}: & 0 \ar[r] & M \ar[r]^{i} & E \ar[r]^{p} & G \ar[r] & 1,}\]
and fix a partial splitting $s_0: N \rightarrow E$ of $\underline{e}$ over $N$. We know that $\alpha^*[\underline{e}]$ can be represented by the lower row in the diagram
\[\xymatrix{\underline{e}: & 0 \ar[r] & M \ar[r]^{i} \ar@{=}[d] & E \ar[r]^{p} & G \ar[r] & 1\\
\underline{e}': & 0 \ar[r] & M \ar[r]^{i'} & P \ar[r]^{p'} \ar[u]_{h} & G' \ar[u]_{\alpha} \ar[r] & 1,}\]
 where the right hand square is a pull-back square. By the universal property of the pull-back, we find a partial splitting $s_0' : N' \rightarrow P$ of $\underline{e}'$ over $N$, such that $h \circ s_0' = s_0 \circ \alpha_0$. Take $q' \in Q'$ and fix an element $x' \in P$ such that $\pi ' \circ p' (x')=q'$. Then $\rho (\alpha^*[\underline{e}])$ can be represented by the derivation $D_1$, mapping $q'$ to the class of the derivation $\widetilde{D}_1(x')$, with $i'(\widetilde{D}_1(x')(n'))=({}^{x'}s_0')(n') s_0'(n')^{-1}$. Straight-forward calculations show that $i(\widetilde{D}_1(x')(n'))=$ $({}^{h(x')} s_0)(\alpha_0(n')) s_0(\alpha_0(n'))^{-1}$.

Now take a look at $\alpha^*(\rho[\underline{e}])$. We can represent $\rho[\underline{e}]$ by
a derivation $D_2$, that maps $q \in Q$ to the class of the derivation $\widetilde{D}_2(x)$, with $\pi \circ p(x)=q$ and $i(\widetilde{D}_2(x)(n))=({}^{x}s_0)(n)s_0(n)^{-1}$ for all $n \in N$. Applying the right-hand map of (\ref{diagr}) to $\rho[\underline{e}]$, we obtain an element that can be represented by a derivation $D_3$, sending $q' \in Q'$ to the class of $\widetilde{D}_2(x) \circ \alpha_0$, where $\pi \circ p(x)=\overline{\alpha}(q')$. Observe that we can choose $x=h(x')$, with $\pi ' \circ p' (x')=q'$. 
Now $i(\widetilde{D}_2(h(x')) (\alpha_0(n')))=({}^{h(x')}s_0)(\alpha_0(n'))s_0(\alpha_0(n'))^{-1}$. It follows that the diagram (\ref{diagr}) commutes, thus $\rho$ is natural. 

\section{More on $tr$}\label{Evens_dinges}

Suppose we take a $G$-module $M$ that is $N$-invariant, so the $G$-action induces a well-defined $Q$-action on $M$. In this case, the equivalence relation on derivations is trivial, and $H^1(N,M)^Q$ can be identified with $\mbox{Hom}_G(N/N',M)$, where $N'$ is the commutator subgroup of $N$. The $G$-action on $N/N'$ is induced by conjugation.

\begin{lemma}\label{lem_evens_ding}
Let $M$ be an $N$-invariant $G$-module, and take $[d]=d \in H^1(N,M)^Q \cong \mbox{Hom}_G(N/N',M)$. Then $\tr[d]=-d_*[\epsilon]$, or equivalently, $\tr[-d]=d_*[\epsilon]$, with 
\[\epsilon:\xymatrix{0 \ar[r] & N/N' \ar[r] & G/N' \ar[r] & Q \ar[r] & 1.}\]
\end{lemma}

\begin{proof}
The partial splitting corresponding to $(-d)$ is $s_{-d} : N \rightarrow M \rtimes G:\;n \mapsto (-d(n),n)$. A representant of $\tr[-d]$ is given by 
\[\underline{e}:\xymatrix{0 \ar[r] & M \ar[r]^-{i} & (M\rtimes G)/s_{-d}(N) \ar[r]^-{p} & Q \ar[r] & 1.}\] 
On the other hand, the extension $d_*[\epsilon]$ can be represented by the second row of the following diagram, where $E$ is the push-out construction.
\[\xymatrix{0 \ar[r] & N/N' \ar[d]^d \ar[r] & G/N' \ar[d]\ar[r] & Q \ar@{=}[d]\ar[r] & 1\\
0 \ar[r] & M \ar[r]^{i'} & E \ar[r]^{p'} & Q \ar[r] & 1}\]

Observe that $E=\big( M \rtimes (G/N') \big) \big/ S$, where $S=\{(-d(n),nN') \in M \rtimes (G/N') \ | \ n \in N\}$. Take $\rho : (M \rtimes G) / s_{-d}(N) \rightarrow E$, mapping $(m,g)s_{-d}(N)$ to $(m,gN')S$. One easily checks that this is an equivalence of extensions, so $\tr[-d]=d_*[\epsilon]$.
\end{proof}

This is exactly the same result as described in Theorem 7.3.1 in \cite{even91-1}, so it means that at least for $M=M^N$, the map $\tr$ we have constructed in this paper coincides with the map induced by the spectral sequence.

\section{Cocycle description}\label{cocycle_description}

Fix a section $\alpha : Q \rightarrow G$ for $\pi$ and let $f_{\alpha} : Q \times Q \rightarrow N$ denote the associated factor set, i.e.\
$\alpha(q)\alpha(q')=f_{\alpha}(q,q')\alpha(qq')$.

\subsection{The map $tr$}
Take an element $[d] \in H^1(N,M)^Q$, the associated partial splitting $s: N \rightarrow M \rtimes G$, and the associated partial semi-direct complement $H = s(N)$. Since $[d]$ is $Q$-invariant, we obtain an exact sequence
\[\xymatrix{0 \ar[r] & M^N \ar[r]^{i \ \ \ } & N_{M \rtimes G}(H) \ar[r] & G \ar[r] & 1,}\] and we choose a section $\widetilde{s} : G \rightarrow N_{M \rtimes G}(H)$, not necessarily a homomorphism, that extends $s$. There is an associated factor set $f_s : G \times G \rightarrow M^N$ such that 
$\widetilde{s}(g)\widetilde{s}(g')=i(f_s(g,g'))\widetilde{s}(gg')$.
We define a section  $\overline{s} : Q \rightarrow N_{M\rtimes G}(H) / H$ of 
\[\underline{e}_H : \xymatrix{0 \ar[r] & M^N \ar[r]^-{\overline{i} \ \ \ \ } & N_{M \rtimes G}(H)/H \ar[r] & Q \ar[r] & 1}\] as $\overline{s}(q)=\widetilde{s}(\alpha(q))H$. 
Now take $q_1, \ q_2 \in Q$. One easily shows that
\begin{eqnarray*}
\overline{s}(q_1)\overline{s}(q_2) & = &  \overline{i}\Big(f_s(\alpha(q_1),\alpha(q_2))-f_s(f_{\alpha}(q_1,q_2),\alpha(q_1q_2))\Big) \overline{s}(q_1q_2).
\end{eqnarray*}
We conclude that $\tr[d]=[F]$, where $F : Q \times Q \rightarrow M^N$ is the cocycle
\[F(q_1,q_2)=f_s(\alpha(q_1),\alpha(q_2))-f_s(f_{\alpha}(q_1,q_2),\alpha(q_1q_2)).\] 
In general, one can choose a section $\widetilde{s} : G \rightarrow N_{M \rtimes G}(H)$ that does not necessarily extend $s$. In this case, the image $\tr[d]$ can be represented by the cocycle $F : Q \times Q \rightarrow M^N$ with
\[F(q_1,q_2)=f_s(\alpha(q_1),\alpha(q_2))-f_s(f_{\alpha}(q_1,q_2),\alpha(q_1q_2))+\overline{i}^{-1}(\widetilde{s}(f_{\alpha}(q_1,q_2))H).\] 
Note that one can use this to show that $\tr \equiv 0$ if $M = M^N$ and $\alpha$ is a homomorphism, taking $\widetilde{s}(g)=(0,g)$.

\begin{remark}
In \cite{rous06-1}, Rousseau gave an ad hoc construction of a map $H^1(N,M)^{Q} \rightarrow H^2(Q,M^N)$ on the cocycle level, rendering the five-term sequence exact. It turns out to coincide with the above cocycle description of our map $\tr$. 
\end{remark}

\subsection{The map $tr$, second description}
It is easy to prove the following lemma.

\begin{lemma} \label{lem:normalization}
Take a derivation $d : N \rightarrow M$ and set $H=\{(d(n),n) \ | \ n \in N\}$, the associated partial semi-direct complement. For $(m,g) \in M \rtimes G$, the following holds: $(m,g) \in N_{M \rtimes G}(H)$ iff $({}^g d-d)(n)=n \cdot m -m$ for all $n \in N$.
\end{lemma}

Take a derivation $d : N \rightarrow M$ such that $[d] \in H^1(N,M)^Q$ and 
fix a normalized map $\eta : Q \rightarrow M$ for which $(^{\alpha(q)} d - d)(n) =n \cdot \eta(q) - \eta(q)$, so that $(\eta(q),\alpha(q))\in N_{M \rtimes G}(H)$ by the previous lemma, where $H$ is defined as before. We claim that a representative cocycle $F : Q \times Q \rightarrow M^N$ of $\tr[d]$ is given as 
\[F(q_1,q_2)=\eta(q_1)+\alpha(q_1) \cdot \eta(q_2)-f_{\alpha}(q_1,q_2) \cdot \eta(q_1q_2)-d(f_{\alpha}(q_1,q_2)).\]
Indeed, we can take a section $\overline{s}:Q \rightarrow N_{M\rtimes G}(H)/H$ of the representative extension
\[\underline{e}_{H} : \xymatrix{0 \ar[r] & M^N \ar[r]^{\overline{i} \ \ \ \ } & N_{M \rtimes G}(H)/H \ar[r] & Q \ar[r] & 1}\]
of $tr[d]$, mapping $q$ to $\overline{s}(q)=(\eta(q),\alpha(q))H$. It is now an easy calculation to see that indeed $\overline{s}(q_1)\overline{s}(q_2)=\overline{i}(F(q_1,q_2)) \overline{s}(q_1q_2)$ for all $q_1$, $q_2 \in Q$, so $F$ is the cocycle associated to $\underline{e}_H$.

Observe that, if $M = M^N$, we can choose $\widetilde{\eta} = 0$. In this case, a representative of $\tr[d]$ is given by $F(q_1,q_2)=-d(f_{\alpha}(q_1,q_2))$. This gives an alternative proof of lemma \ref{lem_evens_ding}. As a corollary, $\tr \equiv 0$ if $M = M^N$ and $\alpha$ is a homomorphism (i.e.\ the sequence of groups is split exact). 

\begin{remark}
Observe that this second cocycle description of our transgression map $\tr$ coincides with an explicit ad hoc description of a map $H^1(N,M)^Q \rightarrow H^2(Q,M^N)$, making the five-term sequence exact, by Guichardet in \cite[$\mathsection{8}$]{guic80-1}. 
\end{remark}
\subsection{The map $\rho$}
Take an element $[f] \in H^2(G,M)_1$ such that $f : G \times G \rightarrow M$ is a cocycle with $f|_{N \times N} =0$. There is a partially split extension 
\[\underline{e}: \xymatrix{0 \ar[r] & M \ar[r] & M \times_f G \ar[r] & G \ar[r] & 1}\]
associated to $f$, where $M \times_f G$ is the set $M \times G$ with group law
\[(m,g)(m',g')=(m + g \cdot m' + f(g,g'),gg').\]
It is trivial to see that $M \times_f N$ is just the semi-direct product $M \rtimes N$. Take a section $s_0 : N \rightarrow M \times_f G$ defined as $s_0(n)=(0,n)$.
Observe that 
\[(m,g)(0,n)(m,g)^{-1}=\big(m+f(g,n)-gng^{-1}\cdot(m+f(g,g^{-1}))+f(gn,g^{-1}),gng^{-1}\big).\]
Remember that $({}^{(m,g)}s_0)(n)=(m,g)s_0(g^{-1}ng)(m,g)^{-1}$, so the derivation $d_{(m,g)}$ associated to ${}^{(m,g)}s_0$ is defined as
\[d_{(m,g)}(n)=m-n\cdot m+f(g,g^{-1}ng)-n \cdot f(g,g^{-1})+f(ng,g^{-1}).\]
Since $f$ is a cocycle, $-n \cdot f(g,g^{-1})+f(ng,g^{-1})=-f(n,g)$, so 
$\widetilde{\rho}_0(\underline{e})(m,g)=d_{(m,g)}$ with 
$d_{(m,g)}(n)=m - n \cdot m + f(g,g^{-1}ng)-f(n,g)$.
It is now easy to see that 
\[\overline{\rho}_0(\underline{e})(q)=[d_g]\] with $\pi(g)=q$ and $d_g(n)=f(g,g^{-1}ng)-f(n,g)$. This gives a complete description of $\rho(\underline{e})=[\overline{\rho}_0(\underline{e})]$.

\subsection{The map $\lambda$}
It is easy to see that, if we fix a cocycle $c : \Out_G(\underline{e}_0) \times \Out_G(\underline{e}_0) \times \Out_G(\underline{e}_0) \rightarrow M^N$ associated to the crossed extension
\[\underline{e}:\xymatrix{0 \ar[r] & M^N \ar[r] & M \rtimes N \ar[r]^{\beta} & \Aut_G(\underline{e}_0) \ar[r]^{\Pi} & \Out_G(\underline{e}_0) \ar[r] & 1,}\]
the image of $[D]$ under $\lambda$ can be represented by $\widetilde{c}$ with $\widetilde{c}(q_1,q_2,q_3)=c(D\overline{s}(q_1),D\overline{s}(q_2) ,D\overline{s}(q_3) )$. This comes from the fact that $\lambda[D]$ can be realized as $(D\overline{s})^*[\underline{e}]$.

On the other hand, we can try to find a direct cocycle description for $\lambda$.
Fix a section $s_2 : H^1(N,M) \rightarrow \mbox{Der}(N,M)$ of the quotient map $\mbox{Der}(N,M) \rightarrow H^1(N,M)$. Since the rows in diagram (\ref{big_diagram}) on page \pageref{big_diagram} are split exact, we can identify $\Aut_G(\underline{e}_0)\cong \mbox{Der}(N, M ) \rtimes G$ and $\Out_G(\underline{e}_0) \cong H^1(N,M) \rtimes Q$. 
Then $\widetilde{s}: \Out_G(\underline{e}_0) \rightarrow \Aut_G(\underline{e}_0)$, defined as $\widetilde{s}([d],q)=(s_2[d],\alpha(q))$
is a section of $\Pi$, so 
$\widetilde{s} \circ D\overline{s}  = (s_2 \circ D, \alpha)$, is a section of $\Phi : \Pi^{-1}(D\overline{s}(Q)) \rightarrow Q$.

We want to compute the map $f : Q \times Q \rightarrow M \rtimes N$, measuring the defect of $(s_2 \circ D,\alpha)$ being a homomorphism. In other words, we want that
\[\beta \circ f(q_1,q_2)\big(s_2D(q_1q_2),\alpha(q_1q_2)\big)=\big(s_2D(q_1),\alpha(q_1)\big)\big(s_2D(q_2),\alpha(q_2)\big).\]
Therefore we measure both the defect of $\alpha$ being a homomorphism (using $f_{\alpha}$), and the defect of $s_2 \circ D$ being a derivation via $\alpha$. Since  $s_2 D (q_1) + {}^{\alpha(q_1)} s_2 D (q_2) - s_2 D (q_1q_2)$ maps to zero in $H^1(N,M)$, it is an inner derivation for all $q_1, \, q_2 \in Q$ and we can fix a map $F' : Q \times Q \rightarrow M$ such that 
\[-\delta^0 (F'(q_1,q_2))=s_2 D (q_1) + {}^{\alpha(q_1)} s_2 D (q_2) - s_2 D (q_1q_2).\]
Using the relation ${}^n d = d + \delta^0(d(n))$, we see that 
\[\beta \circ f(q_1,q_2)=\Big(-\delta^0(F'(q_1,q_2) + \big(s_2 \circ D(q_1q_2)\big) \big(f_{\alpha}(q_1,q_2)\big)), f_{\alpha}(q_1,q_2)\Big),\] and we can choose 
\[f(q_1,q_2)=\Big( F'(q_1,q_2) + \big(s_2 \circ D(q_1q_2)\big) \big(f_{\alpha}(q_1,q_2)\big), f_{\alpha}(q_1,q_2)\Big).\] 

The cocycle $c : Q \times Q \times Q \rightarrow M^N$ of the crossed extension is now defined by \[{}^{\widetilde{s} \circ D\overline{s}(q_1)} f(q_2,q_3) f(q_1,q_2q_3) = i_0c(q_1,q_2,q_3) f(q_1,q_2) f(q_1q_2,q_3),\] where $i_0 : M^N \rightarrow M \rtimes N$ is the obvious embedding (for the correspondence, see \cite[IV.5]{brow82-1}). Using the fact that $f_{\alpha}$ satisfies a non-abelian ``cocycle condition'' (for the definition, see \cite[IV.6]{brow82-1}), and that $s_2 \circ D$ takes values in $\mbox{Der}(N,M)$, together with the definition of $F'$ and the definition of the $G$-action on $\mbox{Der}(N,M)$, we can compute that \begin{equation} \label{third_cocycle}c(q_1,q_2,q_3)=c'(q_1,q_2,q_3) + \big({}^{\alpha(q_1q_2)}s_2D(q_3)\big)(f_{\alpha}(q_1,q_2)),\end{equation} with 
\begin{eqnarray*}
c'(q_1,q_2,q_3 )&= & \alpha(q_1) \cdot F'(q_2,q_3)  -F'(q_1q_2,q_3) \\ && 
\ \ \ +F'(q_1,q_2q_3) - F'(q_1,q_2),
\end{eqnarray*}
which resembles a coboundary expression. The last term in (\ref{third_cocycle}) is given by taking the derivation $s_2 D(q_3) \in \mbox{Der}(N,M)$, letting $\alpha(q_1q_2)$ act on it by the usual $G$-action, and evaluating the resulting derivation in $f_{\alpha}(q_1,q_2) \in N$.

\section{Example: The Heisenberg groups}\label{Heisenberg}

We illustrate the 7-term exact sequence for the Heisenberg groups $G_k$ with trivial coefficient module $\Z$. The group $G_k$ has presentation
\[G_k=\left\langle a, \ b, \ c \ | \ [a,b]=c^k, \ [a,c]=[b,c]=1\right\rangle.\] 
Set $N=Z(G_k)=\langle c \rangle \cong \Z$, so $Q = G_k / N = \langle \overline{b},  \overline{c} \rangle \cong \Z^2$, where $\overline{b}$ and $\overline{c}$ are the images of $b$ and $c$ under the projection map. We obtain the group extension
\[\underline{\epsilon}_k : \xymatrix{0 \ar[r] & N \ar[r] & G_k \ar[r] & Q \ar[r] & 0.}\]

To give an explicit description of the exact sequence, it is important to understand the cohomology groups that appear. We also want to know which cocycles we can choose as group generators. 
It is known that $H^1(Q,\Z)$ is the free abelian group on generators $[f_a]$ and $[f_b]$, with $f_a(\overline{a})=1=f_b(\overline{b})$ and $f_a(\overline{b})=0=f_b(\overline{a})$. Furthermore, there is an isomorphism $H^1(G_k, \Z) \cong \Z^2$, and we can choose generators  $[\widetilde{f}_a]$ and $[\widetilde{f}_b]$ for $H^1(G_k, \Z)$, with $\widetilde{f}_a(a)=1=\widetilde{f}_b(b)$, $\widetilde{f}_a(b)=0=\widetilde{f}_b(a)$ and $\widetilde{f}_a(c)=0=\widetilde{f}_b(c)$. 
Since $N=Z(G_k)$ and $\Z$ is the trivial module, $H^1(N,\Z)$ is invariant under the action of $Q$, so $H^1(N,\Z)^Q\cong \Z$ with generator $[f]$, with $f(c)=1$.   
It is known that the cohomology group $H^2(Q,\Z)$ is isomorphic to $\Z$, with generator $[\underline{\epsilon}_1]$.
We also have the relation $H^2(G_k, \Z) \cong \Z^2 \oplus \Z_k$, and Hartl gives an explicit isomorphism $D : \Z^2 \oplus \Z_k \rightarrow H^2(G_k,\Z)$ in the example in \cite{hart96-1} on p.~410. The generators of $H^2(G_k,\Z)$ are $D(1,0,0)$, $D(0,1,0)$ and $D(0,0,1)$.
Since $H^2(N,\Z)=0$, $H^2(G_k,\Z)_1=H^2(G_k, \Z)$. Last of all, $H^1(Q,H^1(N,\Z))$ is a free abelian group on the generators $[f_1]$ and $[f_2]$, with $f_1(\overline{a})=f=f_2(\overline{b})$ and $f_1(\overline{b})=0=f_2(\overline{a})$.

The exact sequence is now of the form
\[\xymatrix{0 \ar[r] & \Z^2 \ar[r]^-{\infl} & \Z^2 \ar[r]^-{\res} & \Z \ar[r]^-{\tr} & \Z \ar[r]^-{\infl} & \Z^2 \oplus \Z_k \ar[r]^-{\rho} & \Z^2 \ar[r] & 0.}\]
One can easily see that $\infl$ is the identical map after identification of the groups with $\Z^2$. 
It is clear that any homomorphism $f : G_k \rightarrow \Z$ takes $c$ to zero, so $\res \equiv 0$. 
From section \ref{Evens_dinges} we know that $\tr[f]=-[\underline{\epsilon}_k]$, since $f : N \rightarrow \Z$ is an isomorphism. It is easy to see that $[\underline{\epsilon}_k]=k [\underline{\epsilon}_1]$, so $\tr(1)=-k$ after identification. 

For the description of the next maps, we make use of the isomorphism in \cite{hart96-1}.
Computing a cocycle for $\underline{\epsilon}_1$, one easily verifies that $\infl[\underline{\epsilon}_1]$ can be represented by a cocycle sending a couple $(a^k b^l c^m , a^{k'} b^{l'} c^{m'})$ to $c^{k'l}$. It follows that $\infl[\underline{\epsilon}_1]$ corresponds to $D(0,0,1)$. This means that $\infl$ is the composition of the projection $p_k: \Z \rightarrow \Z_k$ and the embedding $i_l$ of the last factor in $\Z^2 \oplus \Z_k$. 

We claim that the last map is the projection $\Z^2 \oplus \Z_k \rightarrow \Z^2$. Indeed, using the formulas in \cite{hart96-1}, one proves that $D(1,0,0)$ maps to the class of $g_1 : Q \rightarrow H^1(N,\Z)$, where $g_1(\overline{a}^{\alpha}\overline{b}^{\beta})$ sends $c$ to $\alpha$. This means that $g_1(\overline{a}) = f$ and $g_1(\overline{b})=0$, so $g_1=f_1$. Analogously, $D(0,1,0)$ is sent to the class of $g_2$, with $g_2(\overline{a})=0$ and $g_2(\overline{b})=f$, so $g_2=f_2$. The element $D(0,0,1)$ is sent to zero. Now it is clear that $\rho : \Z^2 \oplus \Z_k \rightarrow \Z^2$ is the projection $p_{12}$ on the first two factors. 

Therefore the exact sequence equals
\[\xymatrix{0 \ar[r] & \Z^2 \ar[r]^-{id} & \Z^2 \ar[r]^{0} & \Z \ar[r]^-{-k} & \Z \ar[r]^-{i_l \circ p_k} & \Z^2 \oplus \Z_k \ar[r]^-{p_{12}} & \Z^2 \ar[r] & 0.}\] 

\section{Splitting sequence of groups with $N$-invariant module} \label{split_extension}
We consider the special case in which the sequence 
\[\xymatrix{1 \ar[r] & N \ar[r] & G \ar[r] & Q \ar[r] & 1}\] splits and the module $M$ is $N$-invariant, so $M^N = M$. Fix a splitting $\alpha : Q \rightarrow G$.

We have seen in section \ref{Evens_dinges} that in this case, $\tr$ will be the zero map. Moreover, we can show that the map $\lambda$ will also be trivial, since we can find a section $\sigma : \Out_G(\underline{e}_0) \rightarrow \Aut_G(\underline{e}_0)$ of $\Pi$, that is a homomorphism. 

Take $\Pi(h,g) \in \Out_G(\underline{e}_0)$ and set $g'=\alpha(\pi(g))$. Now there exists a map $h'$ such that $(h',g') \in \Aut_G(\underline{e}_0)$ and $\Pi(h',g')=\Pi(h,g)$.
One can show that $h'$ is the unique map with these properties. 

Define $\widetilde{\alpha} : \Out_G(\underline{e}_0) \rightarrow \Aut_G(\underline{e}_0)$, mapping $\Pi(h,g)$ to $(h',g')$ as in the above. It is easy to check that this is a homomorphism. 
Observe that this means that the crossed extension
\[\xymatrix{0 \ar[r] & M^N \ar[r] & M \rtimes N \ar[r]^{\beta} & \Aut_G(\underline{e}_0) \ar[r]^{\Pi} & \Out_G(\underline{e}_0) \ar[r] & 1}\] is equivalent to zero, so $\lambda \equiv 0$.

These results are compatible with results for the sequence induced by the spectral sequence. 

\bibliography{ref}
\bibliographystyle{ref}

\end{document}